\documentclass[a4paper,11pt]{article} 

\usepackage{geometry}
\usepackage[utf8]{inputenc}
\usepackage{amssymb,amsmath,amsfonts,amsthm}
\usepackage[matrix,arrow,curve]{xy} 
\usepackage{verbatim,url}
\usepackage{setspace}
\usepackage{hyperref}

\usepackage{graphicx}
\usepackage{longtable}
\usepackage{tikz}
\usepackage{tikz-cd}

\newtheorem{theorem}{Theorem}[section]
\newtheorem{proposition}[theorem]{Proposition}
\newtheorem{corollary}[theorem]{Corollary}
\newtheorem{lemma}[theorem]{Lemma}
\newtheorem{definition}[theorem]{Definition}
\newtheorem{remark}[theorem]{Remark}
\newtheorem{example}[theorem]{Example}

\theoremstyle{remark}
\newtheorem{rem}[theorem]{Remark}
\newtheorem{comprem}[theorem]{Computational Remark}
\newtheorem{notation}[theorem]{Notation}

\newtheorem*{theorem-non}{Theorem}

\newcommand{\Aut}{\operatorname{Aut}}
\newcommand{\GL}{\operatorname{GL}}
\newcommand{\Res}{\operatorname{Res}}
\newcommand{\ord}{\operatorname{ord}}

\newcommand{\tr}{\operatorname{tr}}
\newcommand{\SL}{\operatorname{SL}}
\newcommand{\Hol}{\operatorname{Hol}}

\title{Mixed Threefolds Isogenous to a Product}
\date{}

\author{Christian Gleissner }

\begin{document}
\maketitle

\begin{abstract}

\bigskip
\noindent
In this paper we study \emph{threefolds isogenous to a product of mixed type} i.e. quotients  
of a product of three compact Riemann surfaces $C_i$ of genus at least two by the action of a finite 
group $G$, which is free, but not diagonal. In particular, we are interested in the systematic 
construction and classification of these varieties. Our main result is the full classification of 
threefolds isogenous to a product of mixed type  with $\chi(\mathcal O_X)=-1$ assuming that 
any automorphism in $G$, which restricts to the trivial element in $\Aut(C_i)$ for some $C_i$, 
is the identity on the product. 
Since the \emph{holomorphic Euler-Poincar\'e-characteristic} of a smooth threefold of general type with ample canonical class is always negative, these examples lie on the boundary, in the sense of \emph{threefold geography}.  
To achieve our result we use techniques from computational group theory. Indeed, we develop a MAGMA algorithm to classify these threefolds for any given value of $\chi(\mathcal O_X)$. 
\end{abstract}

\bigskip

\section*{Introduction}

\bigskip
\bigskip
A central aspect of algebraic geometry is to study symmetries of algebraic varieties 
i.e. their automorphisms. The goal in mind is to produce new and interesting 
varieties as quotients of known and well understood ones by groups of automorphisms.
This idea was already applied very successfully by the classical geometers such as Godeaux et al. 
(see \cite{Godeaux}). 
The quotients that we are interested in were introduced by Catanese in \cite{cat}:

\noindent
A \emph{variety $X$ isogenous to a product of curves} is a quotient of a product of smooth complex projective 
curves $C_i$ of genus at least two by the free action of a finite group of automorphisms:
\[
X=\big(C_1 \times \ldots \times C_n\big)/G.
\]

\noindent
The freeness of the action implies that $X$ is a smooth projective variety of general type with ample 
canonical class $K_X$.   
In the last ten years several authors studied different aspects of the two-dimensional case extensively. 
In particular, surfaces isogenous to a product with \emph{holomorphic Euler-Poincar\'e-characteristic} 
$\chi(\mathcal O_X)= 1$ were completely classified (see \cite{bauer, polizzi, pen10} et al). 
However, a systematic treatment of the higher-dimensional case was still missing until the author, 
in collaboration with Davide Frapporti \cite{FG16}, started to study these varieties in dimension three under 
the assumption that the action is unmixed i.e. each automorphism acts diagonally:
\[
G \leq \Aut(C_1) \times \Aut(C_2) \times \Aut(C_3).
\]
Their main result is the full classification of these threefolds with $\chi(\mathcal O_X)=-1$ and a faithful 
$G$-action on each curve $C_i$. Geographically, these examples are extremal, since the holomorphic 
Euler-Poincar\'e-characteristic of a smooth projective threefold of general type with ample canonical class is negative in contrast to the surface case where it is positive (cf. \cite{miyaoka}).
In this paper we investigate the algebraically more involved \emph{mixed case}, where $G$ is acting 
non-diagonally. Our aim is to derive an algorithm, i.e. a finite procedure 
to classify mixed threefolds isogenous to a product for a fixed value of $\chi(\mathcal O_X)$.  
In particular, we want to determine the Galois groups $G$ and the 
Hodge numbers $h^{p,q}(X)$ of these varieties. Similarly to the unmixed case, 
the technical condition that the induced actions $\psi_i \colon G_i \to \Aut(C_i)$
of the groups 
\[
G_i:= G \cap  
\big[\Aut(C_1 \times \ldots \times \widehat{C_i} \times \ldots \times C_n) \times \Aut(C_i) \big]
\]
have trivial kernels $K_i$ is imposed: we say that the $G$-action is \emph{absolutely faithful}. 
This condition allows us to derive an effective bound for the order of the Galois group $G$, in terms of 
$\chi(\mathcal O_X)$, as well as numerical 
constraints on the genera of the curves $C_i$ and the branching data $T_i$ of the covers $C_i \to C_i/G_i$. 
These restrictions are very strong, indeed they allow only finitely many combinations, which are determined 
in the first step of our algorithm. Then, 
using Riemann's existence theorem, we classify all possible group actions $\psi_i \colon G_i \to \Aut(C_i)$. 
In the next step, we determine all suitable groups $G$ containing $G_i$ that allow us to define a 
free and mixed $G$-action on the product  $C_1 \times C_2 \times C_3$ via the maps $\psi_i$. 
Finally, we compute the \emph{Hodge numbers} of the quotients 
\[
X=(C_1 \times C_2 \times C_3)/G.
\]
This is achieved by analysing the induced representation of $G$ on the \emph{Dolbeault cohomology groups} 
$H^{p,q}(C_1 \times C_2 \times C_3)$.  
The classification procedure is computationally hard and cannot be carried out by hand, even 
in the boundary case $\chi(\mathcal O_X)= -1$. For this reason, the computer algebra system 
MAGMA \cite{magma} is used. 
We want to point out that the strategy of our algorithm differs slightly according to the index of the diagonal subgroup 
\[
G^0 := G \cap \big[\Aut(C_1) \times \Aut(C_2) \times \Aut(C_3)\big]
\]
in $G$. Since the quotient  $G/G^0$ embeds naturally into the permutation group $\mathfrak S_3$ of the coordinates 
of the product (cf. Proposition \ref{semiCurve}), there are three cases:
\[
G/G^0 \simeq \mathbb Z_2, \quad \quad G/G^0 \simeq \mathfrak A_3 \quad \quad 
\makebox{and} \quad \quad G/G^0 \simeq \mathfrak S_3.
\]
We call them index two, index three and index six case, respectively. 
In the index two case we can assume that the second and the third curve are isomorphic $C_2 \simeq C_3$, whereas 
all three curves $C_i$ are isomorphic if the index is three or six. \\ 
Running our implementation for the boundary value $\chi(\mathcal O_X)= -1$ we obtain the following theorems. 
\bigskip

\begin{theorem}\label{indexzwei}
\footnote{We refer to Notation \ref{not} for the definition of the groups in the table.}
Let $X$ be a threefold isogenous to a product of curves of mixed type. Assume that the action of $G$ is absolutely faithful, 
$\chi(\mathcal O_X)=-1$ and the index of $G^0$ in $G$ is two. 
Then, the tuple
\[
[G,T_1,T_2, h^{3,0}(X), h^{2,0}(X), h^{1,0}(X), h^{1,1}(X), h^{1,2}(X),d]
\]
appears in the table below. Conversely, each row in the table is realized by a family of threefolds, 
depending on $d$ parameters, which is obtained by an absolutely faithful $G$-action. 
\end{theorem}

\bigskip

\begin{center}
{\scriptsize 
\begin{tabular}{|l|l|l|l|l|l|l|l|l|l|l|}
\hline
No. & $G$ & Id &$T_1$ & $T_2$ & $h^{3,0}$ & $h^{2,0}$ & $h^{1,0}$ &$h^{1,1}$ & $h^{1,2}$ & $d$\\ 
\hline
\hline
$1$ & $\mathbb Z_2 $ & $\langle 2, 1 \rangle$  & $[2;-]$ & $[2;-]$ & $5$ & $7$ & $4$ & $18$ & $24$ & $6$\\
\hline
$2$ & $\mathbb Z_4 $ & $\langle 4, 1 \rangle$  & $[0;2^2,4^2]$ & $[2;-]$ & $4$ & $4$ & $2$ & $11$ & $16$ & $4$\\
$3$ & $\mathbb Z_4 $ & $\langle 4, 1 \rangle$  & $[2;-]$ & $[0;2^6]$ & $7$ & $7$ & $2$ & $14$ & $22$ & $6$\\
$4$ & $\mathbb Z_4 $ & $\langle 4, 1 \rangle$  & $[2;-]$ & $[1;2^2]$ & $5$ & $6$ & $3$ & $14$ & $20$ & $5$\\
$5$ & $\mathbb Z_2^2 $ & $\langle 4, 2 \rangle$  & $[2;-]$ & $[0;2^6]$ & $5$ & $5$ & $2$ & $14$ & $20$ & $6$\\
$6$ & $\mathbb Z_2^2 $ & $\langle 4, 2 \rangle$  & $[2;-]$ & $[1;2^2]$ & $4$ & $5$ & $3$ & $14$ & $19$ & $5$\\
\hline
$7$ & $\mathfrak S_3 $ & $\langle 6, 1 \rangle$  & $[2;-]$ & $[0;3^4]$ & $4$ & $4$ & $2$ & $12$ & $17$ & $4$\\
$8$ & $\mathbb Z_6 $ & $\langle 6, 2 \rangle$  & $[2;-]$ & $[0;3^4]$ & $5$ & $5$ & $2$ & $12$ & $18$ & $4$\\
\hline
$9$ & $\mathbb Z_8$ & $\langle 8, 1 \rangle$  & $[2;-]$ & $[0;2^2,4^2]$ & $5$ & $5$ & $2$ & $12$ & $18$ & $4$\\
$10$ & $\mathbb Z_4 \times \mathbb Z_2$ & $\langle 8, 2 \rangle$  & $[2;-]$ & $[0;2^2,4^2]$ & $5$ & $5$ & $2$ & $12$ & $18$ & $4$\\
$11$ & $\mathbb Z_4 \times \mathbb Z_2$ & $\langle 8, 2 \rangle$  & $[0;2^2,4^2]$ & $[1;2^2]$ & $3$ & $2$ & $1$ & $7$ & $11$ & $3$\\
$12$ & $\mathbb Z_4 \times \mathbb Z_2$ & $\langle 8, 2 \rangle$  & $[2;-]$ & $[0;2^5]$ & $5$ & $5$ & $2$ & $12$ & $18$ & $5$\\
$13$ & $\mathbb Z_4 \times \mathbb Z_2$ & $\langle 8, 2 \rangle$  & $[2;-]$ & $[0;2^5]$ & $6$ & $6$ & $2$ & $12$ & $19$ & $5$\\
$14$ & $\mathcal D_4$ & $\langle 8, 3 \rangle$  & $[2;-]$ & $[0;2^2,4^2]$ & $4$ & $4$ & $2$ & $12$ & $17$ & $4$\\
$15$ & $\mathcal D_4$ & $\langle 8, 3 \rangle$  & $[0;2^2,4^2]$ & $[1;2^2]$ & $3$ & $2$ & $1$ & $8$ & $12$ & $3$\\
$16$ & $\mathcal D_4$ & $\langle 8, 3 \rangle$  & $[1;2]$ & $[0;2^6]$ & $3$ & $2$ & $1$ & $8$ & $12$ & $4$ \\
$17$ & $\mathcal D_4$ & $\langle 8, 3 \rangle$  & $[1;2]$ & $[1;2^2]$ & $2$ & $2$ & $2$ & $8$ & $11$ & $3$\\
$18$ & $\mathcal D_4$ & $\langle 8, 3 \rangle$  & $[2;-]$ & $[0;2^5]$ & $4$ & $4$ & $2$ & $11$ & $16$ & $5$\\
$19$ & $\mathcal D_4$ & $\langle 8, 3 \rangle$  & $[2;-]$ & $[0;2^5]$ & $5$ & $5$ & $2$ & $12$ & $18$ & $5$\\
$20$ & $Q$ & $\langle 8, 4 \rangle$  & $[2;-]$ & $[0;2^2,4^2]$ & $6$ & $6$ & $2$ & $12$ & $19$ & $4$\\
$21$ & $\mathbb Z_2^3$ & $\langle 8, 5 \rangle$  & $[2;-]$ & $[0;2^5]$ & $4$ & $4$ & $2$ & $12$ & $17$ & $5$\\
\hline
$22$ & $\mathcal D_5$ & $\langle 10, 1 \rangle$  & $[2;-]$ & $[0;5^3]$ & $4$ & $4$ & $2$ & $10$ & $15$ & $3$\\
$23$ & $\mathbb Z_{10}$ & $\langle 10, 2 \rangle$  & $[2;-]$ & $[0;5^3]$ & $4$ & $4$ & $2$ & $12$ & $17$ & $3$\\
\hline
$24$ & $Dic12$ & $\langle 12, 1 \rangle$  & $[2;-]$ & $[0;2^2,3^2]$ & $6$ & $6$ & $2$ & $12$ & $19$ & $4$\\
$25$ & $Dic12$ & $\langle 12, 1 \rangle$  & $[2;-]$ & $[0;3,6^2]$ & $5$ & $5$ & $2$ & $10$ & $16$ & $3$\\
$26$ & $\mathbb Z_{12}$ & $\langle 12, 2 \rangle$  & $[2;-]$ & $[0;2^2,3^2]$ & $5$ & $5$ & $2$ & $12$ & $18$ & $4$\\
$27$ & $\mathbb Z_{12}$ & $\langle 12, 2 \rangle$  & $[2;-]$ & $[0;3,6^2]$ & $4$ & $4$ & $2$ & $12$ & $17$ & $3$\\
$28$ & $\mathcal D_6$ & $\langle 12, 4 \rangle$  & $[2;-]$ & $[0;2^2,3^2]$ & $4$ & $4$ & $2$ & $12$ & $17$ & $4$\\
$29$ & $\mathcal D_6$ & $\langle 12, 4 \rangle$  & $[2;-]$ & $[0;3,6^2]$ & $4$ & $4$ & $2$ & $10$ & $15$ & $3$\\
$30$ & $\mathcal D_6$ & $\langle 12, 4 \rangle$  & $[2;-]$ & $[0;2^2,3^2]$ & $4$ & $4$ & $2$ & $11$ & $16$ & $4$\\
$31$ & $\mathbb Z_3 \times \mathbb Z_2^2$ & $\langle 12, 5 \rangle$  & $[2;-]$ & $[0;2^2,3^2]$ & $5$ & $5$ & $2$ & $12$ & $18$ & $4$\\
$32$ & $\mathbb Z_3 \times \mathbb Z_2^2$ & $\langle 12, 5 \rangle$  & $[2;-]$ & $[0;3,6^2]$ & $4$ & $4$ & $2$ & $12$ & $17$ & $3$\\
\hline
$33$ & $\mathbb Z_{16}$ & $\langle 16, 1 \rangle$  & $[2;-]$ & $[0;2,8^2]$ & $4$ & $4$ & $2$ & $12$ & $17$ & $3$\\
$34$ & $\mathbb Z_2^2 \rtimes_\varphi \mathbb Z_4$ & $\langle 16, 3 \rangle$  & $[1;2]$ & $[0;2^2,4^2]$ & $3$ & $2$ & $1$ 
		 & $6$ & $10$ & $2$\\
$35$ & $\mathbb Z_2^2 \rtimes_\varphi \mathbb Z_4$ & $\langle 16, 3 \rangle$  & $[0;2^2,4^2]$ & $[0;2^5]$ & $4$ & $2$ & $0$ 
     & $6$ & $11$ & $3$\\
$36$ & $\mathbb Z_2^2 \rtimes_\varphi \mathbb Z_4$ & $\langle 16, 3 \rangle$  & $[0;2^2,4^2]$ & $[0;2^5]$ & $3$ & $1$ & $0$ 
     & $5$ & $9$ & $3$\\
$37$ & $\mathbb Z_2^2 \rtimes_\varphi \mathbb Z_4$ & $\langle 16, 3 \rangle$  & $[0;2^2,4^2]$ & $[0;2^5]$ & $4$ & $2$ 
     & $0$ & $5$ & $10$ & $3$\\
$38$ & $\mathbb Z_2^2 \rtimes_\varphi \mathbb Z_4$ & $\langle 16, 3 \rangle$  & $[0;2^2,4^2]$ & $[0;2^5]$ & $4$ & $2$ & $0$ 
     & $7$ & $12$ & $3$\\
$39$ & $\mathbb Z_2^2 \rtimes_\varphi \mathbb Z_4$ & $\langle 16, 3 \rangle$  & $[1;2]$ & $[0;2^5]$ & $3$ & $2$ & $1$ 
     & $6$ & $10$ & $3$\\
$40$ & $\mathbb Z_2^2 \rtimes_\varphi \mathbb Z_4$ & $\langle 16, 3 \rangle$  & $[1;2]$ & $[0;2^5]$ & $4$ & $3$ & $1$ 
     & $6$ & $11$ & $3$\\
$41$ & $\mathbb Z_4 \rtimes_\varphi \mathbb Z_4$ & $\langle 16, 4 \rangle$  & $[1;2]$ & $[0;2^2,4^2]$ & $3$ & $2$ & $1$ 
     & $6$ & $10$ & $2$\\
$42$ & $\mathbb Z_4 \rtimes_\varphi \mathbb Z_4$ & $\langle 16, 4 \rangle$  & $[1;2]$ & $[0;2^2,4^2]$ & $4$ & $3$ & $1$ 
     & $6$ & $11$ & $2$\\
$43$ & $\mathbb Z_8 \times \mathbb Z_2 $ & $\langle 16, 5 \rangle$  & $[2;-]$ & $[0;2,8^2]$ & $4$ & $4$ & $2$ & $12$ & $17$ & $3$\\
$44$ & $M_{16}$ & $\langle 16, 6 \rangle$  & $[2;-]$ & $[0;2,8^2]$ & $5$ & $5$ & $2$ & $10$ & $16$ & $3$\\
\hline
\end{tabular}
}
\end{center}

\begin{center}
{\scriptsize 
\begin{tabular}{|l|l|l|l|l|l|l|l|l|l|l|}
\hline
No. & $G$ & Id &$T_1$ & $T_2$ & $h^{3,0}$ & $h^{2,0}$ & $h^{1,0}$ &$h^{1,1}$ & $h^{1,2}$ & $d$\\ 
\hline
\hline
$45$ & $\mathcal D_8$ & $\langle 16, 7 \rangle$  & $[2;-]$ & $[0;2,8^2]$ & $4$ & $4$ & $2$ & $10$ & $15$ & $3$\\
$46$ & $\mathcal D_8$ & $\langle 16, 7 \rangle$  & $[2;-]$ & $[0;2^3,4]$ & $4$ & $4$ & $2$ & $11$ & $16$ & $4$\\
$47$ & $SD16$ & $\langle 16, 8 \rangle$  & $[2;-]$ & $[0;2,8^2]$ & $4$ & $4$ & $2$ & $12$ & $17$ & $3$\\
$48$ & $SD16$ & $\langle 16, 8 \rangle$  & $[2;-]$ & $[0;4^3]$ & $4$ & $4$ & $2$ & $11$ & $16$ & $3$\\
$49$ & $SD16$ & $\langle 16, 8 \rangle$  & $[0;2,4,8]$ & $[1;2^2]$ & $3$ & $2$ & $1$ & $7$ & $11$ & $2$\\
$50$ & $SD16$ & $\langle 16, 8 \rangle$  & $[2;-]$ & $[0;2^3,4]$ & $5$ & $5$ & $2$ & $11$ & $17$ & $4$\\
$51$ & $Dic16$ & $\langle 16, 9 \rangle$  & $[2;-]$ & $[0;2,8^2]$ & $6$ & $6$ & $2$ & $10$ & $17$ & $3$\\
$52$ & $Dic16$ & $\langle 16, 9 \rangle$  & $[2;-]$ & $[0;4^3]$ & $5$ & $5$ & $2$ & $11$ & $17$ & $3$\\
$53$ & $\mathbb Z_2^2 \times \mathbb Z_4$ & $\langle 16, 10 \rangle$  & $[0;2^2,4^2]$ & $[0;2^5]$ & $3$ & $1$ & $0$ 
     & $5$ & $9$ & $3$\\
$54$ & $\mathcal D_4 \times \mathbb Z_2$ & $\langle 16, 11 \rangle$  & $[2;-]$ & $[0;2^3,4]$ & $4$ & $4$ & $2$ 
     & $11$ & $16$ & $4$\\
$55$ & $\mathcal D_4 \times \mathbb Z_2$ & $\langle 16, 11 \rangle$  & $[0;2^2,4^2]$ & $[0;2^5]$ & $3$ & $1$ 
     & $0$ & $6$ & $10$ & $3$\\
$56$ & $Q \times \mathbb Z_2$ & $\langle 16, 12 \rangle$  & $[2;-]$ & $[0;4^3]$ & $5$ & $5$ & $2$ & $11$ & $17$  & $3$ \\
$57$ & $\mathcal D_4 \ast_{\phi} \mathbb Z_4$ & $\langle 16, 13 \rangle$  & $[2;-]$ & $[0;4^3]$ & $4$ & $4$ 
     & $2$ & $11$ & $16$ & $3$\\
$58$ & $\mathcal D_4 \ast_{\phi} \mathbb Z_4$ & $\langle 16, 13 \rangle$  & $[2;-]$ & $[0;2^3,4]$ & $5$ & $5$ 
     & $2$ & $11$ & $17$ & $4$\\
\hline
$59$ & $Dic20$ & $\langle 20, 1 \rangle$  & $[2;-]$ & $[0;2,5,10]$ & $6$ & $6$ & $2$ & $10$ & $17$ & $3$\\
$60$ & $\mathbb Z_{20}$ & $\langle 20, 2 \rangle$  & $[2;-]$ & $[0;2,5,10]$ & $4$ & $4$ & $2$ & $12$ & $17$ & $3$\\
$61$ & $\mathcal D_{10}$ & $\langle 20, 4 \rangle$  & $[2;-]$ & $[0;2,5,10]$ & $4$ & $4$ & $2$ & $10$ & $15$ & $3$\\
$62$ & $\mathbb Z_2^2 \times \mathbb Z_5$ & $\langle 20, 5 \rangle$  & $[2;-]$ & $[0;2,5,10]$ & $4$ & $4$ & $2$ 
     & $12$ & $17$ & $3$\\
\hline
$63$ & $Dic24$ & $\langle 24, 4 \rangle$  & $[2;-]$ & $[0;3,4^2]$ & $5$ & $5$ & $2$ & $11$ & $17$ & $3$\\
$64$ & $\mathfrak S_3 \times \mathbb Z_4$ & $\langle 24, 5 \rangle$  & $[2;-]$ & $[0;3,4^2]$ & $4$ & $4$ & $2$ & $11$ & $16$ & $3$\\
$65$ & $\mathfrak S_3 \times \mathbb Z_4$ & $\langle 24, 5 \rangle$  & $[2;-]$ & $[0;2^3,3]$ & $5$ & $5$ & $2$ & $11$ & $17$ & $4$\\
$66$ & $\mathcal D_{12}$ & $\langle 24, 6 \rangle$  & $[2;-]$ & $[0;2^3,3]$ & $4$ & $4$ & $2$ & $11$ & $16$ & $4$\\
\hline
$67$ & $Dic12 \times \mathbb Z_2$ & $\langle 24, 7 \rangle$  & $[2;-]$ & $[0;2,6^2]$ & $5$ & $5$ & $2$ & $10$ & $16$ & $3$\\
$68$ & $Dic12 \times \mathbb Z_2$ & $\langle 24, 7 \rangle$  & $[2;-]$ & $[0;2,6^2]$ & $6$ & $6$ & $2$ & $10$ & $17$ & $3$\\
$69$ & $Dic12 \times \mathbb Z_2$ & $\langle 24, 7 \rangle$  & $[2;-]$ & $[0;3,4^2]$ & $5$ & $5$ & $2$ & $11$ & $17$ & $3$\\
$70$ & $\mathbb Z_3 \rtimes_\varphi \mathcal D_4$ & $\langle 24, 8 \rangle$  & $[2;-]$ & $[0;2^3,3]$ & $5$ & $5$ & $2$ 
     & $11$ & $17$ & $4$\\
$71$ & $\mathbb Z_3 \rtimes_\varphi \mathcal D_4$ & $\langle 24, 8 \rangle$  & $[2;-]$ & $[0;3,4^2]$ & $4$ & $4$ & $2$ 
     & $11$ & $16$ & $3$\\
$72$ & $\mathbb Z_3 \rtimes_\varphi \mathcal D_4$ & $\langle 24, 8 \rangle$  & $[2;-]$ & $[0;2,6^2]$ & $4$ & $4$ & $2$ 
     & $10$ & $15$ & $3$\\
$73$ & $\mathbb Z_3 \rtimes_\varphi \mathcal D_4$ & $\langle 24, 8 \rangle$  & $[2;-]$ & $[0;2,6^2]$ & $4$ & $4$ & $2$ 
     & $12$ & $17$ & $3$\\
$74$ & $\mathbb Z_6 \times \mathbb Z_4$ & $\langle 24, 9 \rangle$  & $[2;-]$ & $[0;2,6^2]$ & $4$ & $4$ & $2$ & $12$ & $17$ & $3$\\
$75$ & $\mathcal D_4 \times \mathbb Z_3 $ & $\langle 24, 10 \rangle$  & $[2;-]$ & $[0;2,6^2]$ & $4$ & $4$ & $2$ & $11$ & $16$ & $3$\\
$76$ & $\mathcal D_4 \times \mathbb Z_3 $ & $\langle 24, 10 \rangle$  & $[2;-]$ & $[0;2,6^2]$ & $5$ & $5$ & $2$ & $10$ & $16$ & $3$\\
$77$ & $\mathcal D_{6} \times \mathbb Z_2$ & $\langle 24, 14 \rangle$  & $[2;-]$ & $[0;2^3,3]$ & $4$ & $4$ & $2$ & $11$ & $16$ & $4$\\
$78$ & $\mathcal D_{6} \times \mathbb Z_2 $ & $\langle 24, 14 \rangle$  & $[2;-]$ & $[0;2,6^2]$ & $4$ & $4$ & $2$ 
     & $10$ & $15$ & $3$\\
$79$ & $\mathbb Z_2^3 \times \mathbb Z_3 $ & $\langle 24, 15 \rangle$  & $[2;-]$ & $[0;2,6^2]$ & $4$ & $4$ & $2$ & $12$ & $17$ & $3$\\
\hline
$80$ & $\mathbb Z_2^2 \rtimes_\varphi \mathbb Z_8 $ & $\langle 32, 5 \rangle$  & $[1;2]$ & $[0;2,8^2]$ & $2$ & $1$ & $1$ 
     & $6$ & $9$ & $1$\\
$81$ & $\mathbb Z_2^3 \rtimes_\varphi \mathbb Z_4$ & $\langle 32, 6 \rangle$  & $[0;2^2,4^2]$ & $[0;2^3,4]$ & $3$ & $1$ & $0$ 
		 & $4$ & $8$ & $2$\\
$82$ & $\mathbb Z_2^3 \rtimes_\varphi \mathbb Z_4$ & $\langle 32, 6 \rangle$  & $[1;2]$ & $[0;2^3,4]$ & $3$ & $2$ & $1$ 
     & $6$ & $10$ & $2$\\
$83$ & $M16 \rtimes_\varphi \mathbb Z_2$ & $\langle 32, 7 \rangle$  & $[1;2]$ & $[0;2^3,4]$ & $3$ & $2$ & $1$ & $6$ & $10$ & $2$\\
$84$ & $\mathcal D_4 \rtimes_\varphi \mathbb Z_4$ & $\langle 32, 9 \rangle$  &  $[0;2^2,4^2]$ & $[0;2^3,4]$ & $3$ & $1$ & $0$ 
		 & $5$ & $9$ & $2$\\
$85$ & $\mathcal D_4 \rtimes_\varphi \mathbb Z_4$ & $\langle 32, 9 \rangle$  &  $[0;2^2,4^2]$ & $[0;2^3,4]$ & $3$ & $1$ 
		 & $0$ & $4$ & $8$ & $2$\\
$86$ & $\mathcal D_4 \rtimes_\varphi \mathbb Z_4 $ & $\langle 32, 9 \rangle$  & $[0;2,4,8]$ & $[0;2^5]$ & $3$ & $1$ 
     & $0$ & $5$ & $9$ & $2$\\
$87$ & $\mathbb Z_4 \rtimes_\varphi \mathbb Z_8$ & $\langle 32, 12 \rangle$  & $[1;2]$ & $[0;2,8^2]$ & $3$ & $2$ & $1$ 
     & $6$ & $10$ & $1$\\
$88$ & $\mathcal D_4 \times \mathbb Z_4$ & $\langle 32, 25 \rangle$  &  $[0;2^2,4^2]$ & $[0;2^3,4]$ & $3$ & $1$ & $0$ 
     & $4$ & $8$ & $2$\\
$89$ & $\mathbb Z_4 \rtimes_\varphi \mathcal D_4$ & $\langle 32, 28 \rangle$  &  $[0;2^2,4^2]$ & $[0;2^3,4]$ & $3$ & 
$1$ & $0$ & $5$ & $9$ & $2$\\		
$90$ & $SD16 \times \mathbb Z_2 $ & $\langle 32, 40 \rangle$  & $[2;-]$ & $[0;2,4,8]$ & $4$ & $4$ & $2$ & $11$ & $16$ & $3$\\
$91$ & $\mathcal D_8 \ast_{\phi} \mathbb Z_4$ & $\langle 32, 42 \rangle$  & $[2;-]$ & $[0;2,4,8]$ & $4$ & $4$ 
     & $2$ & $11$ & $16$ & $3$\\
$92$ & $\Hol(\mathbb Z_8)$ & $\langle 32, 43 \rangle$  & $[2;-]$ & $[0;2,4,8]$ & $4$ & $4$ & $2$ & $10$ & $15$ & $3$\\
$93$ & $SD16 \rtimes_\varphi \mathbb Z_2 $ & $\langle 32, 44 \rangle$  & $[2;-]$ & $[0;2,4,8]$ & $5$ & $5$ & $2$ 
     & $10$ & $16$ & $3$\\
\hline
$94$ & $2O$ & $\langle 48, 28 \rangle$  &  $[2;-]$ & $[0;3^2,4]$ & $5$ & $5$ & $2$ & $11$ & $17$ & $3$\\
$95$ & $\GL(2,\mathbb F_3)$ & $\langle 48, 29 \rangle$  &  $[2;-]$ & $[0;3^2,4]$ & $4$ & $4$ & $2$ & $11$ & $16$ & $3$ \\
\hline
$96$ & $\SL(2,3) \times \mathbb Z_2$ & $\langle 48, 32 \rangle$  &  $[2;-]$ & $[0;3^2,4]$ & $5$ & $5$ & $2$ & $11$ & $17$ & $3$\\
$97$ & $\SL(2,3) \rtimes_\varphi \mathbb Z_2$ & $\langle 48, 33 \rangle$  &  $[2;-]$ & $[0;3^2,4]$ & $4$ & $4$ & $2$ 
     & $11$ & $16$ & $3$\\
$98$ & $Dic24 \rtimes_\varphi \mathbb Z_2 $ & $\langle 48, 37 \rangle$  &  $[2;-]$ & $[0;2,4,6]$ & $4$ & $4$ 
     & $2$ & $11$ & $16$ & $3$\\
$99$ & $\mathcal D_4 \times \mathfrak S_3 $ & $\langle 48, 38 \rangle$  &  $[2;-]$ & $[0;2,4,6]$ & $4$ & $4$ 
     & $2$ & $10$ & $15$ & $3$\\
$100$& $\mathcal D_4 \rtimes_\varphi \mathfrak S_3 $ & $\langle 48, 39 \rangle$  &  $[2;-]$ & $[0;2,4,6]$ & $5$ & $5$ & $2$ 
		 & $10$ & $16$ & $3$\\
$101$& $\mathbb Z_6 \rtimes_\varphi \mathcal D_4$ & $\langle 48, 43 \rangle$  &  $[2;-]$ & $[0;2,4,6]$ & $4$ & $4$ & $2$ 
		 & $11$ & $16$ & $3$\\
\hline
\end{tabular}
}
\end{center}

\begin{center}
{\scriptsize 
\begin{tabular}{|l|l|l|l|l|l|l|l|l|l|l|}
\hline
No. & $G$ & Id &$T_1$ & $T_2$ & $h^{3,0}$ & $h^{2,0}$ & $h^{1,0}$ &$h^{1,1}$ & $h^{1,2}$ & $d$\\ 
\hline
\hline
$102$& $\mathfrak S_4 \times \mathbb Z_4 $ & $\langle 96, 186 \rangle$  & $[0;2^2,4^2]$ & $[0; 2, 4, 6 ]$ & $3$ & $1$ & $0$ 
		 & $3$ & $7$ & $1$\\
$103$& $\GL(2,\mathbb F_3)  \times \mathbb Z_2 $ & $\langle 96, 189 \rangle$  & $[2;-]$ & $[0; 2, 3, 8 ]$ & $4$ & $4$ 
		 & $2$ & $11$ & $16$ & $3$\\
$104$& $(Q \times \mathbb Z_2) \rtimes_\varphi \mathfrak S_3  $ & $\langle 96, 190 \rangle$  & $[2;-]$ & $[0; 2, 3, 8 ]$ 
     & $5$ & $5$ & $2$ & $10$ & $16$ & $3$\\
$105$& $2O \rtimes_\varphi \mathbb Z_2  $ & $\langle 96, 192 \rangle$  & $[2;-]$ & $[0; 2, 3, 8 ]$ & $4$ & $4$ & $2$ 
     & $11$ & $16$ & $3$\\
$106$& $\GL(2,\mathbb F_3)  \rtimes_\varphi \mathbb Z_2 $ & $\langle 96,193 \rangle$  & $[2;-]$ & $[0;2,3,8]$ & $4$ & $4$ 
     & $2$ & $10$ & $15$ & $3$\\
$107$& $\GL(2,\mathbb Z_4)$ & $\langle 96, 195 \rangle$  & $[0;2^2,4^2]$ & $[0; 2, 4, 6 ]$ & $3$ & $1$ & $0$ & $4$ & $8$ & $1$\\
\hline
\end{tabular}
}
\end{center}

\bigskip
\bigskip

\noindent
{\bf Explanation.} The table above is organized in the following way:  
\begin{itemize}
\item 
The first column enumerates the examples. 
\item
The second column reports the Galois group (see Notation \ref{not} for the definition of the groups).
\item
The third column provides the MAGMA identifier of the Galois group: $\langle a,b \rangle$ denotes the $b^{th}$ 
group of order $a$ in the \emph{Database of Small Groups} (see \cite{magma}).
\item
The types $T_i=[g_i';m_{i,1}, \ldots ,m_{i,r_i}]$ in column 4 and 5 report the branching data 
of the covers $C_i \to C_i/G_i$.  Here $g_i'$ is the genus of $C_i/G_i$ and the $m_{i,j}$'s are 
the branching orders. 
The types are written in a simplified way: e.g.  $[0;2,2,4,4]$ is abbreviated by $[0;2^2,4^2]$.  
\item
The remaining columns report the Hodge numbers $h^{p,q}(X)$ and the number $d$ of parameters of the families 
(see Remark \ref{moduli}).  
\end{itemize}

\bigskip

\begin{theorem}\label{indexdreiundsechs}
\makebox{$$}
\begin{itemize}
\item[a)]
In the index three case there exists a unique group $G$ yielding a one-dimensional family of threefolds $X$ isogenous 
to a product with $\chi(\mathcal O_X)=-1$. 
The group $G \simeq  \mathbb Z_9 \rtimes_\varphi  \mathbb Z_3$ has MAGMA id $\langle 27,4 \rangle$ and 
the Hodge numbers are: 
\[
h^{3,0}(X)=4, \quad h^{2,0}(X)=2, \quad h^{1,0}(X)=0, \quad  h^{1,1}(X)=5 \quad \makebox{and} \quad  h^{1,2}(X)=10.
\]
\item[b)]
There is no group acting absolutely faithful and freely
on a product of curves such that the quotient $X$ has $\chi(\mathcal O_X)=-1$ and the index of 
$G^0$ in $G$ is six.
\end{itemize}
\end{theorem}

\bigskip

\noindent 
The above theorems complete the classification of threefolds isogenous to a product with absolutely faithful 
$G$-action and $\chi(\mathcal O_X) =-1$. 
For the full list of examples in the unmixed case we refer to  \cite[Theorem 0.1]{FG16}.
It is remarkable that there are no rigid examples in our classification neither in the mixed,
nor in the unmixed case. In contrast, there are six examples of rigid surfaces $S$ isogenous to a product with 
$\chi(\mathcal O_S)=1$ (see \cite{bauer}); one of them was originally discovered by Beauville (see \cite{Be83}). 
To obtain rigid, three-dimensional examples with $\chi(\mathcal O_X)=-1$, we need to allow non-trivial 
kernels $K_i$. Indeed, modifying Beauvilles construction, we are able to give an example of such a threefold (see Example \ref{RigidExample}). It would be interesting to find more, or even try to classify them completely. 
In addition, we also provide an example of a threefold $X$ isogenous to a product, where the index of 
$G^0$ in $G$ is six, the $G$-action is not absolutely faithful and $\chi(\mathcal O_X)=-1$ 
(see Example \ref{Index6Examples}).

\bigskip
\noindent 
The paper is organized in the following way:
in Section \ref{generalities} we introduce varieties isogenous to a product and explain 
their basic properties. Section \ref{GactProdCurves} 
is dedicated to the structure of mixed group actions on a product of
three curves. In Section \ref{group_descr} we define the algebraic datum of a mixed threefold $X$ isogenous 
to a product. Based on that, we show in Section \ref{hodgenumbers} how to determine the Hodge numbers of a threefold
$X$ isogenous to a product from an algebraic datum of $X$. 
In Section \ref{bounds_smooth} we develop an algorithm to classify threefolds isogenous 
to a product of mixed type for a fixed value of $\chi(\mathcal O_X)$, that are obtained by an absolutely faithful group action, and present our main result: the classification of these varieties in the case $\chi(\mathcal O_X)=-1$. 

\begin{notation}\label{not}
Throughout the paper all varieties are defined over the field of complex numbers
and the standard notation from complex algebraic geometry is used, see for example \cite{GriffH}. 
Moreover, we have the following notations and definitions from group theory:

\begin{itemize}
\item
The cyclic group of order $n$ is denoted by 
$\mathbb Z_n$, the dihedral group of order $2n$ by $\mathcal D_n$ and the symmetric and alternating group on $n$ letters by $\mathfrak S_n$ and $\mathfrak A_n$, respectively.
\item 
The quaternion group of order $8$ is defined as $Q:= \langle -1,i,j,k  ~ \big\vert ~i^2=j^2=k^2=ijk =-1 \rangle$. 
\item
The groups $\GL(n,\mathbb F_q)$ and $\SL(n,\mathbb F_q)$ are the general linear and special linear groups 
of $n \times n$ matrices over the field $\mathbb F_q$.
\item
The holomorph $\Hol(G)$ of a group $G$ is the semi-direct product $G \rtimes_{id} \Aut(G)$.
\item
Let $G_1$ and $G_2$ be groups with isomorphic subgroups  $U_i \leq Z(G_i)$ and let $\phi \colon U_1 \to U_2$ be an isomorphism. 
The central product $G_1 \ast_{\phi} G_2$ is defined as the quotient of the direct product  
$G_1 \times G_2$ by the normal subgroup 
$\lbrace (g_1,g_2) \in U_1 \times U_2 ~ \big\vert ~ \phi(g_1) g_2 = 1_{G_2} \rbrace$.
\item
The dicyclic group of order $4n$ is $Dic4n:=\langle a,b,c ~ \big\vert ~ a^n=b^2=c^2=abc \rangle$.
\item
The  semidihedral group of order $2^n$ is 
$SD2^n:=\langle a,b ~ \big\vert ~ a^{2^{(n-1)}}=b^2=1, ~ bab = a^{2^{(n-2)}-1} \rangle$.
\item
The group $M_{16}$ of order $16$ is 
$M_{16}:=\langle a,b ~ \big\vert ~ a^8=b^2=e, ~ bab^{-1}=a^5 \rangle$.
\item
The binary octahedral group of order $48$ is 
$2O:=\langle a,b,c ~ \big\vert ~ a^4=b^3=c^2=abc \rangle$.
\end{itemize}
\end{notation}

\bigskip

\noindent 
{\bf Acknowledgements.} The author thanks Ingrid Bauer, Fabrizio Catanese, Davide Frapporti, Roberto Pignatelli and Sascha Weigl for several suggestions and useful mathematical discussions.

\newpage

\section{Generalities}\label{generalities}

\bigskip
\noindent
In this section we introduce the objects that we study and state some of their basic properties.

\bigskip
\begin{definition}
A complex algebraic variety $X$ \emph{isogenous to a product of curves} is a quotient 
\[
(C_1 \times \ldots \times C_n)/G
\]
of a product of compact Riemann surfaces $C_i$ of genus at least two by a finite group acting freely on the product. 
\end{definition}

\noindent 
The freeness of the $G$-action implies that $X$ is a projective manifold 
of general type with ample canonical class $K_X$. It also allows us to derive formulas for the Chern invariants 
$\chi(\mathcal O_X)$, $e(X)$ and $K_X^n$ in terms of the group order and the genera of the Riemann surfaces 
(see \cite[Proposition 1.2]{FG16}):    
\[
\chi(\mathcal O_X)= \frac{(-1)^n}{|G|} \prod_{i=1}^n \big( g(C_i)-1 \big), \quad 
\quad e(X)= 2^n  \chi(\mathcal O_X)
\quad \makebox{and} \quad K_X^n= (-1)^n n! ~ 2^n  \chi(\mathcal O_X).
\]
To study group actions on a product of compact Riemann surfaces $C_i$ of genus at least two, it is 
important to understand the structure of the automorphism group of the product. 
This group has a simple description in terms of the automorphism groups $\Aut(C_i)$ 
of the factors:

\begin{proposition}[{\cite[Corollary 3.9]{cat}}]\label{semiCurve}   
Let $D_1,\ldots,D_k$ be pairwise non isomorphic compact Riemann surfaces of genus at least two, then   
\[
\Aut(D_1^{n_1} \times \ldots \times D_k^{n_k})= 
\big(\Aut(D_1)^{n_1} \rtimes \mathfrak S_{n_1}\big) \times \ldots \times 
\big(\Aut(D_k)^{n_k} \rtimes \mathfrak S_{n_k}\big) 
\]
for all positive integers $n_i$. 
\end{proposition}

\noindent 
Motivated by the proposition, we give the following definition: 

\begin{definition}
Let $G$ be a subgroup of $\Aut(C_1 \times \ldots \times C_n)$, where $g(C_i) \ge 2$. Then we define 
\[ 
G_i:=G \cap  \big(\Aut(C_1 \times \ldots \times \widehat{C_i} \times \ldots \times C_n) \times \Aut(C_i) \big)
\quad \quad \makebox{for all} \quad \quad  1 \leq i \leq n.
\]
\end{definition}

\noindent 
Note that the elements of $G_i$ are precisely those automorphisms that can be restricted to $C_i$.
We write $\psi_i\colon G_i \to \Aut(C_i)$ for the restriction homomorphism 
and denote its kernel by $K_i$. Clearly, the diagonal subgroup
\[
G^0:= G \cap \big(\Aut(C_1)\times \ldots \times \Aut(C_n)\big) \trianglelefteq G 
\] 
is equal to the intersection of the groups $G_i$ and 
the quotient $G/G^0$ embeds naturally in the permutation group $\mathfrak S_n$ of the 
factors of the product. If $G/G^0$ is trivial, we say that the action on the product is \emph{unmixed} and otherwise 
\emph{mixed}. Similarly,  the quotient of the product by $G$ is called of \emph{unmixed} or \emph{mixed type}, respectively. In this paper, unless otherwise stated, we will consider the mixed case  
 (see \cite{FG16} for the unmixed).

\bigskip\bigskip

\section{Mixed Actions}\label{GactProdCurves}

\bigskip

\noindent
An unmixed action of a finite group $G$ on a product of Riemann surfaces is given by 
\[
g(x_1, \ldots , x_n)= \big(\psi_1(g) x_1, \ldots , \psi_n(g) x_n\big), 
\quad \quad \makebox{for all} \quad \quad g \in G.
\]
In this section we show that after conjugation with a
suitable automorphism in 
\[
\Aut(C_1)\times \ldots \times \Aut(C_n)
\]
there are analogous formulas describing a mixed $G$-action in terms of the maps $\psi_i$.
Such a description, i.e. a \emph{normalized form} of the action, is of great importance for the following reasons:

\begin{itemize}
\item
it allows us to study the geometric properties of the quotient 
$(C_1 \times \ldots \times C_3)/G$
using Riemann surface theory,
\item
the formulas defining the normal
form can be used to construct an action of an abstract finite group $G$ on a product of compact
Riemann surfaces starting from suitable subgroups $G_i \leq G$  and group actions  $\psi_i \colon G_i \to \Aut(C_i)$. 
\end{itemize}

\noindent
For simplicity, we assume 
that $n=3$, but similar results can be obtained in all dimensions.  
For $n=2$ we refer the reader to \cite[Proposition 3.16 ii)]{cat}.
According to the index of $G^0$ in $G$, there are three sub-cases of the mixed case:
\[
G/G^0 \simeq \mathbb Z_2, \quad \quad G/G^0 \simeq \mathfrak A_3 \quad \quad \makebox{and} \quad \quad G/G^0 \simeq \mathfrak S_3.
\]
We call them index two, index three and index six case, respectively.

\bigskip
\noindent
{\bf Convention:} 
in the index two case we can assume that $C_2 \simeq C_3$. In 
the index three and six case it holds $C_1 \simeq C_2 \simeq C_3$. 
If we specialize in one of these cases, we may write  
$D \times C^2$ or $C^3$ instead of $C_1 \times C_2 \times C_3$.

\begin{proposition}\label{normalaction}
Let $G$ be a subgroup of the automorphism group of a product of three compact Riemann surfaces 
and $\nu \colon G \to G/G^0 \leq \mathfrak S_3$ be the projection map. 
\begin{itemize}
\item[i)]
In the index two case we fix an element 
$\delta \in G$ of the form  $\delta(x,y,z)=(\delta_1 x, \delta_3 z, \delta_2 y)$, i.e. $\nu(\delta) = (2,3)$.
Then, after 
conjugating with the automorphism 
$\xi(x,y,z) := (x, y,\delta_3 z)$, it holds
\[
\psi_3(g)= \psi_2(\delta g \delta^{-1}) \quad \makebox{for all} \quad  g\in G^0
\]
and the action is given by the formulas 
\begin{itemize}
\item[$\bullet$]
$\delta (x,y,z) = \big(\psi_1(\delta )x, z ,\psi_2(\delta^2) y \big)$ 
\item[$\bullet$]
$g(x,y,z) = \big(\psi_1(g)x,\psi_2(g)y,\psi_2(\delta g \delta^{-1})z\big) \quad $ for all $g \in G^0$. 
\end{itemize}
\item[ii)]
In the index three case we fix an element 
$\tau \in G$ of the form 
$\tau(x,y,z)=( \tau_2 y, \tau_3 z, \tau_1 x )$, i.e. 
$\nu(\tau) = (1,3,2)$. Then, after 
conjugating with the automorphism $\epsilon(x,y,z):=(x,\tau_2y,\tau_2 \tau_3 z)$,
it holds
\[
\psi_2(g)=\psi_1(\tau g \tau^{-1}) \quad \makebox{and} \quad \psi_3(g)= \psi_1(\tau^2 g \tau^{-2}) 
\quad \makebox{for all} \quad  g\in G^0 
\]
and the action is given by the formulas 
\begin{itemize}
\item[$\bullet$]
$\tau (x,y,z) = \big( y,z,\psi_1(\tau^3 ) x \big)$ 
\item[$\bullet$]
$g(x,y,z)= \big(\psi_1(g)x, \psi_1(\tau g \tau^{-1})y,\psi_1(\tau^2 g \tau^{-2})z\big)$ for all $g \in G^0$.
\end{itemize}
\item[iii)]
In the index six case we fix an element 
$\tau \in G$ of the form 
$\tau(x,y,z)=( \tau_2 y, \tau_3 z, \tau_1 x )$, i.e. 
$\nu(\tau) = (1,3,2)$.  Then, after 
conjugating with the automorphism 
$\epsilon(x,y,z):=(x,\tau_2y,\tau_2 \tau_3 z)$,
it holds 
\[
\psi_2(h)=\psi_1(\tau h \tau^{-1}) \quad \makebox{and} \quad  \psi_3(k)= \psi_1(\tau^2 k \tau^{-2})
\]
for all $h \in G_2$ and $k\in G_3$ and the action is given by the formulas 
\begin{itemize}
\item[$\bullet$]
$\tau(x,y,z) = \big( y,z,\psi_1(\tau^3) x \big)$ 
\item[$\bullet$]
$g(x,y,z) = \big(\psi_1(g)x, \psi_1(\tau g \tau^{-1})y,\psi_1(\tau^2 g \tau^{-2})z\big)$ 
\item[$\bullet$]
$f(x,y,z) = \big( \psi_1(f)x, \psi_1(\tau f \tau^{-2}) z, \psi_1(\tau^2 f \tau^{-1}) y \big)$ 
\end{itemize}
for all $g \in G^0$ and $f \in G_1\setminus G^0$. 
\end{itemize}
\end{proposition}

\bigskip
\noindent 
Since the proof of the Proposition is just a calculation, we skip it.

\bigskip
\noindent 
{\bf Convention:} 
from now on we assume that a subgroup 
$G \leq \Aut(C_1 \times C_2 \times C_3)$ is embedded  
in \emph{normal form} for a fixed choice of $\delta$ or $\tau$, respectively.

\bigskip
\noindent 
As already mentioned, a very important observation is that the formulas from Proposition \ref{normalaction}
provide a way to define mixed group actions on a product of three compact Riemann surfaces:

\bigskip

\begin{proposition}\label{constructaction}
Let $G$ be a finite group with a normal subgroup $G^0$ such that 
$G/G^0$ is isomorphic to $\mathbb Z_2$, $\mathfrak A_3$ or $\mathfrak S_3$.
Let $\nu\colon G \to G/G^0$ be the quotient map. 

\begin{itemize}
\item[i)]
In the index two case, let 
$\psi_1\colon G \to \Aut(D)$ and $\psi_2\colon G^0 \to \Aut(C)$ be group actions on compact Riemann surfaces 
with kernels $K_i$ such that 
\[
K_1 \cap K_2 \cap \delta K_2 \delta^{-1}  = \lbrace 1_G \rbrace
\] 
for an element $\delta \in G \setminus G^0$. Then 
the  formulas from Proposition \ref{normalaction} $i)$ define an embedding 
$G \hookrightarrow \Aut(D \times C^2)$.

\item[ii)]
In the index three case, let $\alpha\colon G/G^0 \to \mathfrak A_3$ be an isomorphism and
$\psi_1\colon G^0 \to \Aut(C)$ be a group action on a compact Riemann surface with kernel $K_1$ such that 
\[
K_1 \cap \tau K_1 \tau^{-1} \cap \tau^2 K_1 \tau^{-2} = \lbrace 1_G \rbrace
\] 
for an element $\tau \in G$ with $(\alpha \circ \nu)(\tau)=(1,3,2)$. Then
the  formulas from Proposition \ref{normalaction} $ii)$ define an embedding 
$G  \hookrightarrow \Aut(C^3)$. 
\item[iii)]
In the index six case, let $\alpha\colon G/G^0 \to \mathfrak S_3$ be an isomorphism.
Define 
\[
G_1:=(\alpha \circ \nu)^{-1}\big(\langle (2,3)\rangle \big)
\]
and let  
$\psi_1\colon G_1 \to \Aut(C)$ be a group action on a compact Riemann surface with kernel $K_1$ such that 
\[
K_1 \cap \tau K_1 \tau^{-1} \cap \tau^2 K_1 \tau^{-2} = \lbrace 1_G \rbrace
\] 
for an element $\tau \in G$ with $(\alpha \circ \nu)(\tau)=(1,3,2)$. 
Then the formulas from Proposition \ref{normalaction} $iii)$ define an embedding 
$G  \hookrightarrow \Aut(C^3)$. 
\end{itemize}
\end{proposition}

\bigskip

\begin{definition}[{cf. \cite[Definition 3.1]{FG16}}]
As in the unmixed case, we say that the $G$-action is \emph{minimal} if 
$K_i \cap K_j$ is trivial for all $i \neq j$ and \emph{absolutely faithful}
if the kernels $K_i$ are trivial. 
\end{definition}

\bigskip
\noindent
We point out that every threefold  isogenous to a product can be obtained by a unique minimal action 
(cf. \cite[Proposition 3.13]{cat}). 
Hence, from now on, we assume that the action of $G$ is minimal. 
Note that the kernels $K_i$ are related: we have 
$K_3 = \delta K_2 \delta^{-1}$ in the index two case, whereas $K_2 = \tau^2 K_1 \tau^{-2}$ and 
$K_3 = \tau K_1 \tau^{-1}$ in the index three and index six case.

\bigskip
\noindent 
In Proposition \ref{constructaction} we described the building data to define a mixed group action
on a product of three curves. Since we want to construct  
threefolds isogenous to a product, we shall give conditions that ensure the freeness of the action on 
the product in terms of the maps $\psi_i$. 
Such conditions are easily deduced from the formulas in Proposition \ref{normalaction} 
(cf. \cite[Proposition 3.16 ii)]{cat} for the two dimensional case). To phrase them in a compact way, the following definition is convenient.

\begin{definition}
Let  $\psi\colon H \to \Aut(C)$  a group action on a Riemann surface.  The stabilizer set $\Sigma \subset H$
of $\psi$ is defined as the set of elements admitting at least one fixed point on $C$. 
\end{definition}

\begin{proposition}\label{freenesscond}
Let $G$ be a subgroup of the automorphism group of a product $C_1 \times C_2 \times C_3$ of compact Riemann surfaces
and $\Sigma_i$ be the stabilizer set of $\psi_i \colon G_i \to \Aut(C_i)$. Then the freeness of the $G$-action is equivalent to:  
\begin{itemize}
\item[a)] index two case
\begin{itemize}
\item[i)]
$\Sigma_1 ~ \cap ~ \Sigma_2 ~ \cap ~ \delta \Sigma_2 \delta^{-1} = \lbrace 1_G \rbrace$ and
\item [ii)] 
for all  $g \in G^0$ with $\delta g \in \Sigma_1$, it holds 
$(\delta g)^2 \notin \Sigma_2$.
\end{itemize}
\item[b)] index three case
\begin{itemize}
\item[i)]
$\Sigma_1 ~ \cap ~ \tau \Sigma_1 \tau^{-1} ~ \cap ~ \tau^2 \Sigma_1\tau^{-2} = \lbrace 1_G \rbrace$ and
\item[ii)]
for all  $g \in G^0$ it holds $(\tau g)^3 \notin \Sigma_1$. 
\end{itemize}
\item[c)]
index six case
\begin{itemize}
\item[i)]
$\Sigma_1 ~ \cap ~ \tau \Sigma_1 \tau^{-1} ~ \cap ~ \tau^2 \Sigma_1 \tau^{-2} = \lbrace 1_G \rbrace$,
\item[ii)]
for all  $g \in G^0$ it holds $(\tau g)^3 \notin \Sigma_1$ and
\item[iii)]
for all  $f \in G_1 \setminus G^0$ with $f \in \Sigma_1$, it holds 
$\tau f^2 \tau^{-1} \notin \Sigma_1$.
\end{itemize}
\end{itemize}
\end{proposition}

\begin{rem}
For completeness, we want to mention that an unmixed action is free if an only the intersection 
$\Sigma_1 ~ \cap ~ \Sigma_2 ~ \cap ~  \Sigma_3$ is trivial.
\end{rem}

\bigskip

\begin{corollary}\label{shortexactnonsplit}
Let $G$ be a subgroup of the automorphism group of a product of three curves. 
\begin{itemize}
\item[a)]
Assume that $G^0 \trianglelefteq G$ is of index six and $G$ is acting freely on the product, then 
the short exact sequence 
\[
1 \longrightarrow G^0 \longrightarrow  G \longrightarrow \mathfrak S_3 \longrightarrow 1
\]
does not split. 
\item[b)]
Assume that $G^0 \trianglelefteq G$ is of index three and condition $i)$ in 
Proposition \ref{freenesscond} $b)$ holds. Then $ii)$ in Proposition \ref{freenesscond} $b)$ is equivalent to 
the condition, that the short exact sequence 
\[
1 \longrightarrow G^0 \longrightarrow  G \longrightarrow \mathfrak A_3 \longrightarrow 1
\]
does not split. 
\end{itemize}
\end{corollary}

\begin{proof}
$a)$
A short exact sequence $1 \longrightarrow G^0 \longrightarrow  G \longrightarrow \mathfrak S_3 \longrightarrow 1$
splits, if and only if there exist elements  $a,b \in G \setminus G^0$ such that $\ord(a)=2$, $\ord(b)=3$ and $aba=b^{-1}$.  
Assume that the sequence splits, then there exist elements $a,b$ as above. Since $b^2 \notin G^0$ we 
can assume that $b=\tau g \in \tau G^0$.
This leads to the contradiction $(\tau g)^3=1 \in \Sigma_1$. 
The proof of $b)$ is similar.
\end{proof}

\bigskip
\bigskip
\section{The algebraic datum}\label{group_descr}

\bigskip
The aim of this section is to give a \emph{group theoretical description} of a threefold 
\[
X=(C_1 \times C_2 \times C_3)/G
\]
isogenous to a product of mixed type. We refer the reader to \cite[Section 3]{FG16} for the unmixed case. 
In the previous section, we worked out a description of mixed actions using the maps
\[
\psi_i \colon G_i \to \Aut(C_i); 
\]
dividing by their kernels $K_i$, we obtain effective actions of the factor groups $G_i/K_i$ on the curves $C_i$, which can be characterized by \emph{Riemann's existence theorem}: 

\begin{theorem}[{\cite[cf. Sections III.3 and III. 4]{mir}}]
An effective action $\psi \colon H \to \Aut(C)$
of a finite group $H$ on a compact Riemann surface $C$ is given and completely described by 
the following 
\begin{itemize}
\item
a compact Riemann surface $C'$,
\item
a finite set $\mathcal B \subset C'$ (the branch points) and 
\item
a surjective homomorphism 
$\eta\colon \pi_1\big(C' \setminus \mathcal B, q_0\big) \to H \quad$
(the monodromy map). 
\end{itemize}
\end{theorem}

\bigskip
\noindent 
Recall that the fundamental group of $C' \setminus \mathcal B$ has a presentation of the form 
\[
\pi_1\big(C' \setminus \mathcal B, q_0\big)= 
\big\langle \gamma_1, \ldots, \gamma_r, \alpha_1,\beta_1,\ldots, \alpha_{g'},\beta_{g'} ~ \big\vert ~
\gamma_1 \cdots \gamma_r \cdot \prod_{i=1}^{g'} [\alpha_i,\beta_i] \big\rangle.
\]
Here, the generators $\gamma_i$ are simple loops around the branch points 
(see Figure \ref{geometric basis}).

\begin{figure}
\begin{center}
\begin{tikzpicture}[scale=0.6, transform shape]
\draw[black, line width=1.0pt] plot [smooth, tension=1] coordinates 
{ (2,0) (6,3) (4,4.5) (3,2.5) (1.5,4) (0,3) (0.5,1.5) (-1.5,2.5) (-2,0.5) (2,0)};
\draw [black,domain=0:90] plot ({4.5 + 0.7*cos(\x)}, {3 + 0.7*sin(\x)});
\draw [black,domain=0:90] plot ({5.3 - 0.7*sin(\x)}, {3.8 - 0.7*cos(\x)});
\draw [black,domain=0:90] plot ({0.8 + 0.7*cos(\x)}, {2.5 + 0.7*sin(\x)});
\draw [black,domain=0:90] plot ({1.6 - 0.7*sin(\x)}, {3.3 - 0.7*cos(\x)});
\draw[line width=0.5pt]
(-0.8,1.7) circle (0.2);
\draw[line width=0.5pt]
(-1.8,1.2) circle (0.2);
\draw[->,>=stealth] (-0.6,1.7)-- (-0.61,1.8);
\draw[->,>=stealth] (-1.6,1.2)-- (-1.61,1.3);
\filldraw [black] (-0.8,1.7) circle (0.5pt);
\filldraw [black] (-1.8,1.2) circle (0.5pt);
\filldraw [black] (1.5,0.5) circle (0.5pt);
\node at (1.9,0.5) {$q_0$};
\draw[line width=0.5pt](1.5,0.5) to [out=100,in=290] (-0.8,1.5);
\draw[line width=0.5pt](1.5,0.5) to [out=140,in=300] (-1.8,1.0);
\node at (-1,0.8) {$\gamma_1$};
\node at (-0.3,1.4) {$\gamma_2$};
\node at (0.8,1.8) {$\alpha_1$};
\node at (2.7,2) {$\beta_1$};
\draw[->,>=stealth] (2.124,1.6)-- (2.179,1.7);
\draw[line width=0.5pt](1.5,0.5) to [out=100,in=300] (1.5,2.6);
\draw[densely dashed](1.5,2.6) to [out=60,in=120] (2.5,2.8);
\draw[line width=0.5pt](2.5,2.75) to [out=270,in=70] (1.5,0.5);
\draw[->,>=stealth] (0.795,2.5)-- (0.76,2.6);
\draw[line width=0.5pt](1.5,0.5) to [out=95,in=180] (1.1,3.5); 
\draw[line width=0.5pt] (1.1,3.5) to [out=0,in=75] (1.5,0.5);
\end{tikzpicture}\caption{generators of $\pi_1$ \label{geometric basis}}
\end{center}
\end{figure}
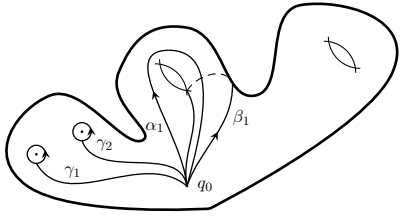

\noindent
Note that the images of the generators of $\pi_1\big(C' \setminus \mathcal B, q_0\big)$ under the monodromy map 
\[
h_i:=\eta(\gamma_i), \quad a_i:=\eta(\alpha_i) \quad \makebox{and} \quad  b_i:=\eta(\beta_i)
\]
generate $H$ and satisfy the relation 
\[
h_1 \cdots  h_r \cdot \prod_{i=1}^{g'}[a_i,b_i] = 1  \quad \quad \quad (\ast).
\]
This provides the motivation for our next definition.

\begin{definition}
Let $m_1,\ldots,m_r \geq 2$  and  $g' \geq 0$ be integers and $H$ be a finite group.  
A \emph{generating vector} for $H$  of type $[g';m_1, \ldots ,m_r]$ is a $(2g'+ r)$-tuple 
\[
(h_1, \ldots , h_r, a_1,b_1, \ldots , a_{g'},b_{g'})
\]
of group elements which generate $H$, satisfy the relation $(\ast)$ and fulfill the condition 
$\ord(h_i)= m_i$  for all $1 \leq i \leq r$. 
\end{definition}

\begin{rem}
Let  $\psi \colon H \to \Aut(C)$ be an effective group action and 
$V=(h_1, \ldots , h_r, a_1,b_1, \ldots , a_{g'},b_{g'})$ be an associated generating vector of $H$.
Since the cyclic groups $\langle h_i \rangle$ and their conjugates provide the non-trivial stabilizers 
of the action, the stabilizer set of $\psi$ can be written as
\[
\Sigma= \bigcup_{h \in H} ~ \bigcup_{i \in \mathbb Z} ~ \bigcup_{j = 1}^r\left \{  hh_j^ih^{-1}\right\}. 
\]
For this reason it makes also sense to refer to $\Sigma$ as the stabilizer set associated to the
generating vector $V$. 
\end{rem}

\noindent 
Similarly to the unmixed case (see \cite[Section 3]{FG16}), we attach to a threefold 
isogenous to a product of mixed type certain algebraic data, reflecting the geometry of the threefold.
We have the groups $G$ and $G^0$, the kernels $K_i$ and 
the embedding $G/G^0 \leq \mathfrak S_3$. In the index three case we choose an element 
$\tau \in G$ with residue class 
$(1,3,2)$; in the index six case we choose elements $\tau, h \in G$ with classes 
$(1,3,2)$ and $(2,3)$.
For each $\overline{\psi_i}\colon G_i/K_i\to \Aut(C_i)$ we can choose 
a generating vector $V_i$ for $G_i/K_i$ of type $T_i$. Not that the latter is not unique, 
only the type $T_i$ is uniquely determined. Since we work in the mixed case, 
the actions $\psi_i$ are related according to Proposition \ref{normalaction}. This 
implies that some of the data is redundant. To keep track of it, we define: 
 
\begin{definition}\label{algebraicdata}
To a threefold $X$ isogenous to a product of mixed type we attach the tuple 
\begin{itemize}
\item
$(G,G^0,K_1,K_2,V_1,V_2)$ in the index two case,
\item
$(G,G^0,K_1,\tau,V_1)$ in the index three case, 
\item
$(G,G^0,K_1,\tau,h,V_1)$
in the index six case,  
\end{itemize}
and call it an \emph{algebraic datum} of $X$.
\end{definition}

\noindent 
Thanks to Riemann's existence theorem, Proposition \ref{constructaction} and \ref{freenesscond} 
we have a way to construct threefolds isogenous 
to a product starting from group theoretical data. For the two dimensional analogue, 
we refer to \cite[Proposition 2.5]{bauer}.  

\bigskip

\begin{proposition}\label{converse}
Let $G$ be a finite group and $G^0 \trianglelefteq G$ be a normal subgroup such that 
$G/G^0 \leq \mathfrak S_3$ and let $\nu\colon G \to G/G^0$ be the quotient map. 
\begin{itemize}
\item[a)]
Assume that $G/G^0 \simeq \mathbb Z_2$. Let 
$\delta \in G \setminus G^0$, $K_1 \trianglelefteq G$ and $K_2 \trianglelefteq G^0$ be normal subgroups such that 
\[
K_1 \cap K_2 = \lbrace 1_G \rbrace \quad \makebox{and} \quad K_2 \cap \delta K_2 \delta^{-1} = \lbrace 1_G \rbrace.
\]
Let $V_1$ be a generating vector for $G/K_1$ and $V_2$ a generating vector for $G^0/K_2$. 
Let $\Sigma_i \subset G$ be the pre-images of the stabilizer sets 
associated to the generating vectors $V_i$ under the quotient maps 
\[
G \to G/K_1 \quad \makebox{and} \quad G^0 \to G^0/K_2. 
\]
Assume that the freeness conditions from Proposition \ref{freenesscond} $a)$ hold. 
Then there exists a threefold $X$ isogenous to a product with algebraic datum 
\[
(G,G^0,K_1,K_2,V_1,V_2).
\]
\item[b)]
Assume that $G/G^0 \simeq \mathbb Z_3$. 
Let $\tau \in G \setminus G^0$ and $K_1 \trianglelefteq G^0$ such that 
\[
K_1 \cap \tau K_1 \tau^{-1}=\lbrace 1_G \rbrace.
\]
Let $V_1$ be a generating vector for $G^0/K_1$ and 
$\Sigma_1 \subset G^0$ be the pre-image of the stabilizer set 
associated to the generating vector $V_1$ under the quotient map
\[
G^0 \to G^0/K_1. 
\]
Assume that the freeness conditions from Proposition \ref{freenesscond} $b)$ hold. 
Then there exists a threefold $X$ isogenous to a product with algebraic datum 
\[
(G,G^0,K_1,\tau,V_1).
\] 
\item[c)]
Assume that $G/G^0 \simeq \mathfrak S_3$. 
Let $\tau, h \in G \setminus G^0$ such that $\tau^2 \notin G^0$ and $h^2 \in G^0$.
Define the subgroup 
\[
G_1:=\langle h, G^0 \rangle \leq G.
\] 
Let  $K_1 \trianglelefteq G_1$ be a normal subgroup such that 
\[ 
K_1 \cap \tau K_1 \tau^{-1}=\lbrace 1_G \rbrace.
\]
Let 
$V_1$ be a generating vector for $G_1/K_1$  and 
$\Sigma_1 \subset G_1$ be the pre-image of the stabilizer set 
associated to the generating vector $V_1$ under the quotient map
\[
G_1 \to G_1/K_1. 
\]
Assume that the freeness conditions from Proposition \ref{freenesscond} $c)$ hold.
Then there exists a threefold $X$ isogenous to a product with algebraic datum $(G,G^0,K_1,\tau,h,V_1)$.

\end{itemize}
\end{proposition}

\noindent 
\begin{rem}\label{moduli}
In Proposition \ref{converse} we actually construct families of threefolds, they depend on   
the choice of the complex structure on the quotient curves $C_i'$ and 
the branch points $\mathcal B_i \subset C_i'$, that is 
\begin{itemize}
\item $3(g_1'+g_2')-6 +r_1+r_2$ parameters in the index two case and 
\item $3g_1'-3 +r_1$ parameters in the index three and index six case, respectively.
\end{itemize}
The integers $g_i'=g(C_i')$ and $r_i=|\mathcal B_i|$ are given in terms of the types 
$T_i=[g_i';m_{i,1}, \ldots ,m_{i,r_i}]$  of the generating vectors $V_i$. 
\end{rem}

\section{The Hodge diamond}\label{hodgenumbers}
 
\noindent 
In this section we explain how to compute the \emph{Hodge numbers} of a threefold 
$X=(C_1 \times C_2 \times C_3)/G$ isogenous to a product of mixed type from an algebraic datum of $X$. 

\bigskip
\noindent 
The idea is to use representation theory: 
the action of $G$ on the product induces representations 
\[
\phi_{p,q} \colon G  \to  \GL\big(H^{p,q}(C_1 \times C_2 \times C_3)\big), \quad \quad
 g \mapsto  [\omega \mapsto (g^{-1})^{\ast}\omega]
\]
on the \emph{Dolbeault cohomology groups},  whose characters are denoted 
by $\chi_{p,q}$.
Since the action is free, the Hodge numbers $h^{p,q}(X)$ are given
 as the dimensions of the 
$G$-invariant parts of the Dolbeault groups $H^{p,q}(C_1 \times C_2 \times C_3)$ i.e. 
\[
h^{p,q}(X)= \frac{1}{|G|} \sum_{g\in G} \chi_{p,q}(g). 
\]
In the same way the maps $\psi_i\colon G_i \to \Aut(C_i)$ induce representations 
$\varphi_i\colon G_i \rightarrow \GL\big(H^{1,0}(C_i) \big)$. Their   
characters $\chi_{\varphi_i}$ can be easily determined from the generating vectors $V_i$ in  
the algebraic datum of $X$, thanks to the formula of \emph{Chevalley-Weil}:
see \cite{ChevWeil} for the original version and \cite[Section 2]{FG16} for the relevant details in our situation. 
What remains to do is to determine the characters $\chi_{p,q}$ in terms of the characters $\chi_{\varphi_i}$.  
For this task we need the description of the mixed action from Proposition \ref{normalaction} and  
\emph{K\"unneth's formula for Dolbeault cohomology}: 

\bigskip

\begin{proposition}[{\cite[p.103-104]{GriffH}}]
There is an isomorphism 
\[
H^{p,q}(C_1 \times C_2 \times C_3)\simeq\bigoplus_{\substack {s_1+s_2+s_3=p \\ t_1 + t_2+t_3 =q}}
H^{s_1,t_1}(C_1) \otimes H^{s_2,t_2}(C_2) \otimes H^{s_3,t_3}(C_3),
\]
induced by the natural projections $p_i \colon  C_1 \times C_2 \times C_3 \to C_i$. 
\end{proposition}

\noindent
Formulas for the restrictions of the characters $\chi_{p,q}$ to the diagonal subgroup $G^0$ are easily derived using the fact that the character of a direct sum of representations is the sum of the characters and the character of a tensor product is equal to the product of the characters.

\begin{theorem}[{\cite[Theorem 3.7]{FG16}}]\label{KunnChar}
For the restrictions of the characters $\chi_{p,q}$ it holds:
\begin{itemize}
\item[i)] 
$\Res^{G}_{G^0}\big(\chi_{1,0}\big) = \chi_{\varphi_1} + \chi_{\varphi_2} + \chi_{\varphi_3}$, 
\item[ii)]
$\Res^{G}_{G^0}\big(\chi_{1,1}\big) = 2 \mathfrak Re (\chi_{\varphi_1} \overline{\chi_{\varphi_2}} + 
\chi_{\varphi_1} \overline{\chi_{\varphi_3}}
+ \chi_{\varphi_2} \overline{\chi_{\varphi_3}}) + 3\chi_{triv}$,
\item[iii)] 
$\Res^{G}_{G^0}\big(\chi_{2,0}\big) = \chi_{\varphi_1} \chi_{\varphi_2} + \chi_{\varphi_1}  \chi_{\varphi_3} + \chi_{\varphi_2}  
\chi_{\varphi_3}$,
\item[iv)] 
$\Res^{G}_{G^0}\big(\chi_{2,1}\big) = \overline{\chi_{\varphi_1}}  \chi_{\varphi_2}  \chi_{\varphi_3}  + 
\chi_{\varphi_1}  \overline{\chi_{\varphi_2}}  \chi_{\varphi_3}
+ \chi_{\varphi_1}  \chi_{\varphi_2}  \overline{\chi_{\varphi_3}} + 2(\chi_{\varphi_1} + \chi_{\varphi_2} + \chi_{\varphi_3})$,
\item[v)] 
$\Res^{G}_{G^0}\big(\chi_{3,0}\big) = \chi_{\varphi_1}  \chi_{\varphi_2}  \chi_{\varphi_3}$.
\end{itemize}
Here, $\chi_{triv}$ denotes the trivial character. 
\end{theorem}

\begin{rem}
Since the representations $\chi_{\varphi_i}$ are defined in terms of the actions $\psi_i$, they
are related to each other in the same way as these actions (see Proposition \ref{normalaction}) 
 e.g.  we have 
\[
\chi_{\varphi_2}(h)=\chi_{\varphi_1}(\tau h \tau^{-1}) \quad \quad \makebox{and} \quad \quad 
\chi_{\varphi_3}(k)=\chi_{\varphi_1}(\tau^2 k \tau^{-2})
\]
for all $h \in G_2$ and $k \in G_3$  in the index six case and similar formulas in the other cases.  
\end{rem}

\noindent 
It remains to determine the values of the characters 
$\chi_{p,q}$ for elements outside of $G^0$. Here we need a simple lemma from linear algebra:

\begin{lemma}\label{tensor_traces}
Let $A,B$ and $C$ be endomorphisms of a finite dimensional vector space $V$, then:   
\begin{itemize}
\item[i)]
the trace of the endomorphism of $V^{\otimes 2}$ given by 
$u \otimes v \mapsto  Av \otimes Bu$
is the trace of $A \circ B$,
\item[ii)]
the trace of the endomorphism of $V^{\otimes 3}$ given by 
$u \otimes v \otimes w \mapsto Av \otimes Bw \otimes  Cu$
is  the trace of $A \circ B \circ C$. 
\end{itemize}
\end{lemma}

\begin{theorem}\label{valuesnotinG0}
The values of the characters $\chi_{p,q}$ for the elements outside of $G^0$ are displayed in the table below:

\bigskip

{\small
\begin{center}\label{chartable} 
\setlength{\tabcolsep}{1pt}
\begin{tabular}{c |c | c | c | c | c}
$$ & $(1,0)$ & $ \quad (1,1) \quad $ & $(2,0)$ & $(2,1)$  & $(3,0)$   \\
\hline
$$ & $$ & $$ & $$ & $$ & $$  \\
$\chi_{p,q}(\delta g)$ & $\chi_{\varphi_1}(\delta g)$ & $1$ & $-\chi_{\varphi_2}\big((\delta g)^2\big) $ & 
$-\overline{\chi_{\varphi_1}(\delta g)}\chi_{\varphi_2}\big((\delta g)^2\big)$ &  
$-\chi_{\varphi_1}(\delta g)\chi_{\varphi_2}\big((\delta g)^2\big)$ \\
$$ & $$ & $$ & $$ & $$ & $$  \\
$\chi_{p,q}(\tau g)$ & $0$ & $0$ & $0$ & 
$0$ &  $\chi_{\varphi_1}\big( (\tau g)^3 \big) $ \\
$$ & $$ & $$ & $$ & $$ & $$  \\
$\chi_{p,q}(\tau^2 g)$ & $0$ & $0$ & $0$ & 
$0$ &  $\chi_{\varphi_1}\big( (\tau^2 g)^3 \big)$ \\
$$ & $$ & $$ & $$ & $$ & $$  \\
$\chi_{p,q}(f)$ & $\chi_{\varphi_1}(f)$ & $1$ & $-\chi_{\varphi_2}(f^2) $ & 
$-\overline{\chi_{\varphi_1}(f)}\chi_{\varphi_2}(f^2)$ &  $-\chi_{\varphi_1}(f)\chi_{\varphi_2}(f^2)$ \\
$$ & $$ & $$ & $$ & $$ & $$ \\ 
\end{tabular}
\end{center}
}
\bigskip

\begin{itemize}
\item
the first row holds for all $\delta g \in \delta G^0$ in the index two case, 
\item
the second and third row holds  for all
$\tau g \in \tau G^0$ and  $\tau^2g \in \tau^2  G^0$ in the index three as well as the index six case and 
\item
the last row holds for all 
$f \in G_1\setminus G^0$ in the index six case. 
\end{itemize}
\end{theorem}

\bigskip

\begin{rem}
Note that the table above gives the values of the characters $\chi_{p,q}$ for all elements in $G\setminus G^0$. 
In the index two and index three case this is clear. In the index six case it follows from
the identities 
\[
G_1\setminus G^0 =\tau (G_2 \setminus G^0) \tau^{-1}=\tau^2 (G_3 \setminus G^0) \tau^{-2}
\]
and the fact that a character is constant under conjugation. 
\end{rem}

\begin{proof}[proof of Theorem \ref{valuesnotinG0}]
Under the natural homomorphism $\nu\colon G \to G/G^0 \leq \mathfrak S_3$ an  
element in 
$G \setminus G^0$ maps to a three cycle or to a transposition. 
For this reason we will prove the theorem just in two cases: 
\[
\makebox{a) for  $\tau g \in \tau G^0 \quad $  i.e.  $ \quad \nu(\tau g)=(1,3,2) \quad $ and \quad b)
for $ \quad \delta g \in \delta G^0 \quad $ i.e.  $ \quad \nu(\delta g)=(2,3)$.}
\]
For the elements contained in $\tau^2 G^0$ the computation is identical to $a)$ and for
$f \in G_1 \setminus G^0$ it is identical to $b)$. 

\noindent
$a)$ The inverse of an element $\tau g \in \tau G^0$ acts on $C^3$ via 
\[
(\tau g)^{-1}(x,y,z)=\big( \psi_1( g^{-1} \tau^{-3})z,\psi_1(\tau g^{-1} \tau^{-1})x,\psi_1(\tau^2 g^{-1} \tau^{-2}) y \big).
\]
Let $\omega=\omega_1 \otimes \omega_2 \otimes \omega_3$ be a pure tensor in $
H^{s_1,t_1}(C) \otimes H^{s_2,t_2}(C) \otimes H^{s_3,t_3}(C)$, 
where 
\[
s_1 + s_2 +s_3 =p \quad \quad \makebox{and} \quad \quad  t_1 + t_2 +t_3 =q. 
\]
Under K\"unneth's isomorphism $\omega$ is mapped to 
$p_1^{\ast}\omega_1 \wedge p_2^{\ast}\omega_2 \wedge p_3^{\ast}\omega_3$. The
pullback of this element via $(\tau g)^{-1}$ is:
\[
\pm ~  p_1^{\ast} \psi_1(\tau g^{-1} \tau^{-1})^{\ast} \omega_2 \wedge
p_2^{\ast} \psi_1(\tau^2 g^{-1} \tau^{-2})^{\ast} \omega_3 \wedge 
p_3^{\ast} \psi_1(g^{-1} \tau^{-3})^{\ast} \omega_1,
\]
where the sign depends on the degrees of the classes $\omega_i$. 
The corresponding tensor 
\[
\pm ~ \psi_1(\tau g^{-1} \tau^{-1})^{\ast}\omega_2 \otimes 
\psi_1(\tau^2 g^{-1} \tau^{-2})^{\ast}\omega_3 \otimes  
\psi_1(g^{-1} \tau^{-3})^{\ast}\omega_1 
\]
is an element in 
\[
H^{s_2,t_2}(C) \otimes H^{s_3,t_3}(C) \otimes H^{s_1,t_1}(C).
\]
We conclude that  $\omega$ and $\big((\tau g)^{-1}\big)^{\ast} \omega$ are contained in 
different direct summands for all 
pairs 
\[
(p,q) \in \big\lbrace (1,0), (1,1), (2,0), (2,1) \big\rbrace.
\]
This implies that the traces of the linear maps 
\[
\big((\tau g)^{-1}\big)^{\ast}\colon H^{p,q}(C^3) \to H^{p,q}(C^3)
\]
are equal to zero for these pairs. In other words $\chi_{p,q}(\tau g)=0$.  
In the case $(p,q)=(3,0)$ the forms $\omega_i$ are all of type $(1,0)$. Therefore, the sign in the formula 
for the pullback of $\omega$ is $+1$ and there is only one summand in the decomposition of $H^{3,0}(C^3)$. 
It holds  
\[
\big((\tau g)^{-1}\big)^{\ast} \omega  = \varphi_1(\tau g\tau^{-1}) \omega_2 \otimes 
\varphi_1(\tau^2 g \tau^{-2})^{\ast}\omega_3 \otimes  
\varphi_1(\tau^{3}g)^{\ast}\omega_1.
\]
We apply Lemma \ref{tensor_traces} \emph{ii)} setting 
$A:= \varphi_1(\tau g \tau^{-1})$, $B:=\varphi_1(\tau^2 g \tau^{-2})$ and
$C:=\varphi_1(\tau^{3} g)$ and obtain
\[ 
\chi_{3,0}(\tau g) = \tr(ABC)=\tr\big(\varphi_1(\tau g))^3\big)=\chi_{\varphi_1}\big( (\tau g)^3 \big).
\]
$b)$ take an element $\delta g \in \delta G^0$ and a pure tensor 
$\omega=\omega_1 \otimes \omega_2 \otimes \omega_3$ 
in 
\[
H^{s_1,t_1}(D) \otimes H^{s_2,t_2}(C) \otimes H^{s_3,t_3}(C) \subset H^{p,q}(D \times C^2). 
\]
The pullback of $\omega$ via $(\delta g)^{-1}$ is
\begin{eqnarray*}
\big((\delta g)^{-1}\big)^{\ast} \omega
&=& \pm ~ \psi_1(g^{-1} \delta^{-1})^{\ast}\omega_1 \otimes 
\psi_2(\delta g^{-1} \delta^{-1})^{\ast}\omega_3 \otimes \psi_2(g^{-1} \delta^{-2})^{\ast}\omega_2. 
\end{eqnarray*}	
This is a tensor in $H^{s_1,t_1}(D) \otimes H^{s_3,t_3}(C) \otimes H^{s_2,t_2}(C)$.
For all pairs $(p,q)$, there is exactly one direct summand of $H^{p,q}(D \times C^2)$ containing 
both $\omega$ and  $\big((\delta g)^{-1}\big)^{\ast} \omega$. This implies that 
the trace of $\big((\delta g)^{-1}\big)^{\ast}$ is equal to the trace of 
the restriction of  $\big((\delta g)^{-1}\big)^{\ast}$ to this 
summand. Using Lemma \ref{tensor_traces} \emph{i)} in the same way as above, we get

\bigskip

{\small
\begin{center}
\begin{tabular}{c|c|c}
$(p,q)$ & invariant summand & $\chi_{p,q}(\delta g)$ \\
\hline
& & \\
$(1,0)$ & $H^{1,0}(D) \otimes H^{0,0}(C) \otimes H^{0,0}(C)$ & $\chi_{\varphi_1}(\delta g)$ \\
$$ & $$ & \\
$(1,1)$ & $H^{1,1}(D) \otimes H^{0,0}(C) \otimes H^{0,0}(C)$ & $1$ \\
$$ & $$ & \\
$(2,0)$ & $H^{0,0}(D) \otimes H^{1,0}(C) \otimes H^{1,0}(C)$ & $-\chi_{\varphi_2}\big((\delta g)^2\big)$ \\
$$ & $$  & \\
$(2,1)$ & $H^{0,1}(D) \otimes H^{1,0}(C) \otimes H^{1,0}(C)$ & $-\overline{\chi_{\varphi_1}(\delta g)}
\chi_{\varphi_2}\big((\delta g)^2\big)$\\
$$ & $$ & \\
$(3,0)$ & $H^{1,0}(D) \otimes H^{1,0}(C) \otimes H^{1,0}(C)$ & $-\chi_{\varphi_1}(\delta g)
\chi_{\varphi_2}\big((\delta g)^2\big)$\\
\end{tabular}
\end{center}
}
\end{proof}

\section{Combinatorics, Bounds and Algorithms}\label{bounds_smooth}

Given a threefold isogenous to a product $X=(C_1 \times C_2 \times C_3)/G$, we consider the following numerical information:
\begin{itemize}
\item
the group order $n:=|G|$,
\item
the orders $k_i:=|K_i|$ of the kernels of the maps  $\psi_i \colon G_i \to \Aut(C_i)$ and
\item
the types $T_i=[g_i';m_{i,1}, \ldots ,m_{i,r_i}]$ (see Section \ref{group_descr})  of the corresponding Galois covers 
\[
C_i \rightarrow C_i/\overline{G_i}, \quad \quad \makebox{where} \quad \quad \overline{G_i}:=G_i/K_i.
\]  
\end{itemize}

\noindent
Note that the collection above determines the genera $g_i:=g(C_i)$ via Hurwitz' formula  
\[
g_i= \frac{|G_i|}{2 k_i} \bigg(2g_i'-2+\sum_{j=1}^{r_i} \frac{m_{i,j}-1}{m_{i,j}} \bigg) + 1,
\]
and therefore also the invariants $\chi(\mathcal O_X)$, $e(X)$ and $K_X^3$ of the threefold $X$  
(see Section \ref{generalities}). In analogy to the definition of an algebraic datum 
(see Definition \ref{algebraicdata}) we define: 

\begin{definition}
The \emph{numerical datum} of a threefold $X$ isogenous to a product is the tuple 

\begin{itemize}
\item
$\mathcal D:=(n,k_1,k_2,T_1,T_2)$ in the index two case and
\item
$\mathcal D:=(n,k_1,T_1)$
in the index three and index six case.
\end{itemize}
\end{definition}

\noindent 
If the action is absolutely faithful
$k_i=1$ for all $1 \leq i \leq 3$. Here, as a convention, we omit writing the $k_i's$. 
Clearly, an algebraic datum $\mathcal A$ of $X$ 
determines the numerical datum $\mathcal D$ of $X$ and we say that the numerical datum 
$\mathcal D$ is realized by the algebraic datum $\mathcal A$. 
We point out that $k_2=k_3$ and $T_2 =T_3$ in the index two case, 
whereas $k_1=k_2=k_3$ and $T_1=T_2=T_3$ in the index three and six case.

\noindent 
In this section we derive \emph{combinatorial constraints} on the numerical data.  
These constraints will imply that 
there are only finitely many possibilities for the numerical data, once the value of   
$\chi(\mathcal{O}_X)$ is fixed. Consequently there can be only finitely many algebraic data
realizing these numerical data. This fact can be turned into an algorithm searching systematically
through all possibilities and thereby classifying all threefolds isogenous to a product with a fixed value of 
$\chi(\mathcal{O}_X)$.

\begin{definition}
We define the function 
$N_{\max}\colon \mathbb N_{\geq 2} \to \mathbb N$,
where 
\[
N_{\max}(g):=\max\big\lbrace|\Aut(C)| ~ \big\vert ~ \makebox{$C$ is a compact Riemann surface with $g(C)=g$}\big\rbrace.
\]
\end{definition}

\noindent
According to Hurwitz' classical result $N_{\max}(g)$ is bounded by $84(g-1)$. However, for many values of $g$, the quantity $N_{\max}(g)$ is actually much smaller. 
Conder's paper \cite{Conder} contains a table that displays all $N_{\max}(g)$ in the range $2 \leq g \leq 301$. 
It is the most comprehensive reference that we found and it will be very useful for our computations.

\begin{proposition}\label{boundgroupord}
Let $X= \big(C_1 \times C_2 \times C_3 \big)/G$ be a threefold isogenous to a product of mixed type 
with numerical datum $\mathcal D$. Then
\[
n \leq \bigg\lfloor \sqrt{- d \cdot \chi(\mathcal O_X) \prod_{i=1}^3\frac{k_i}{\Theta_{min}(T_i)}} \bigg\rfloor
\quad \makebox{in the general case and} \quad  
n \leq \lfloor 42 \sqrt{-d \cdot 42 \chi(\mathcal O_X)} 
\]
if the action is absolutely faithful. Here 
\[
\Theta_{min}(T_i):=\begin{cases}
  1/42,  & \text{if} \quad g_i'=0\\
  1/2, & \text{if} \quad g_i'=1\\
	2g_i'-2, & \text{if}  \quad g_i' \geq 2\\ 
\end{cases}, 
\quad \quad \makebox{for} \quad \quad 
	T_i=[g_i'; m_{i,1},\ldots, m_{i,r_i}]
\]
and the parameter $d$  is defined as $32$ in the index two case and as $216$ in the index three and index six case.
\end{proposition}

\noindent 
For a proof we refer to \cite[Proposition 4.4 and Corollary 4.5]{FG16}, which are 
the analogous statements in the unmixed case; indeed setting $d=8$, the formulas in 
the proposition become the bounds in the unmixed case. 

\noindent
Note that in the absolutely faithful case, the bound for $n=|G|$ is solely in terms of $\chi(\mathcal O_X)$. 
Such a result can also be derived in the general case:
let $X= \big(C_1 \times C_2 \times C_3 \big)/G$ be a threefold  isogenous to a product of mixed type, then $X$ is covered by an unmixed threefold 
\[
X^0:= \big(C_1 \times C_2 \times C_3 \big)/G^0 \quad \quad \makebox{and it holds} \quad \quad 
\chi(\mathcal O_{X^0})= \big\vert G/G^0 \big\vert \chi(\mathcal O_{X}).
\]	
According to \cite[Proposition 4.6]{FG16}:
\[
|G^0| \leq 84^6 \chi(\mathcal O_{X^0})^2 \quad \quad \makebox{which yields the bound} \quad \quad 
n=|G| \leq 84^6 \big\vert G/G^0 \big\vert^3 \chi(\mathcal O_X)^2. 
\]
Unfortunately, even in the simplest case, when the holomorphic Euler-Poincar\'e-characteristic is $-1$ this 
bound is too large to be useful from the computational point of view.  
It would be interesting to understand if there exists a significantly better bound for $n=|G|$ in terms of 
$\chi(\mathcal{O}_X)$.

\begin{proposition}\label{bounds1}
Let $X$ be a  threefold isogenous to a product,
with algebraic datum $\mathcal D$. 
Then
\begin{itemize}\setlength{\itemsep}{.3\baselineskip}
\item[i)] $k_i ~ \big\vert  ~ \big( g_{[i+1]} - 1 \big) \big(g_{[i+2]} - 1 \big)$,
\item[ii)] $m_{i,j} ~ \big\vert ~  \big( g_{[i+1]} - 1 \big) \big(g_{[i+2]} - 1 \big)$,
\item[iii)]  $\displaystyle{\big( g_i- 1 \big) ~ \big\vert ~  \chi(\mathcal O_X) \frac{n}{k_i}}$,
\item[iv)] $r_i \leq \dfrac{4d_i k_i\big( g_i - 1 \big)}{n} - 4g_i' + 4$,
\item[v)] $ m_{i,j} \leq  4g_i + 2$,
\item[vi)]
$\displaystyle{ g_i' \leq 1  - \frac{d_i k_i  
\chi(\mathcal O_X)}{\big( g_{[i+1]} - 1 \big) \big( g_{[i+2]}- 1 \big)} \leq 1  - d_i  \chi(\mathcal O_X)}$. 
\item[vii)]
$n/(k_i d_i) \leq N_{\max}(g_i)$ 
\end{itemize}
Here, $[\,\cdot \,]$ denotes the residue $\mathrm{mod }\ 3$ and 
\begin{itemize}
\item
$d_1=1$ and $d_2=d_3=2$ in the index two case, 
\item
$d_i=3$ for all $i$  in the index three and index six case. 
\end{itemize}
\end{proposition}

\noindent 
Also here we obtain the analogous constraints from unmixed case setting $d_i=1$ for all $i$ 
(cf. \cite[Proposition 4.8]{FG16}). 
As an immediate consequence of Proposition \ref{boundgroupord} and Proposition \ref{bounds1} we conclude that  
 there are only finitely many 
algebraic data of threefolds $X$ isogenous to a product with 
\[
\mathbb \chi(\mathcal O_X) \leq -1.
\]
For completeness, we want to state the following useful, but trivial Remark: 

\begin{remark}\label{roots}
\makebox{$$}
\begin{itemize}
\item[a)]
In the index two case $\displaystyle{g_2=\sqrt{\frac{-n \cdot \chi(\mathcal O_X)}{g_1-1}} + 1}$.
\item[b)]
In the index three and index six case $g_1=\sqrt[3]{-n \cdot \chi(\mathcal O_X)} + 1$.
\end{itemize}
\end{remark}

\noindent
The combinatorial constraints that we found enable us to give an algorithm to 
classify threefolds isogenous to 
a product with a fixed value of $\chi(\mathcal O_X)$. 
Since the bound for the group order 
is very large in the general case, a complete classification, even 
with the help of a computer and just for small values of $\chi(\mathcal O_X)$, seems to be out of reach.
On the other hand, if the group action is assumed to be absolutely faithful, then the bound 
drops significantly and a full classification, at least for $\chi(\mathcal O_X)=-1$, 
is possible. 
For this reason, we restrict ourselves to the absolutely faithful case.  
The exact strategy that we follow in our algorithm differs slightly 
according to the index of $G^0$ in $G$. 
Our MAGMA implementation is based on the code given in \cite[Appendix]{BCGP12}. 
We point out that the program relies heavily on \emph{MAGMA's Database of Small Groups} (see \cite{magma}), which contains:
\begin{itemize}
\item  all groups of order up to 2000, excluding the groups of order 1024,
\item  the groups whose order is a product of at most 3 primes,
\item  the groups of order dividing $p^6$ for $p$  prime, 
\item  the groups of order $p^k q$, where $p^k$ is a prime-power dividing 
$2^8$, $3^6$, $5^5$ or $7^4$ and $q$ is a prime different from $p$.
\end{itemize}

\noindent 
Since the full code is very long, we just explain the strategy. 
\footnote{The interested reader can find the full code at
 \url{http://www.staff.uni-bayreuth.de/~ bt300503/.}}

\bigskip

\noindent
{\bf Input:} 
A value $\chi$ for the holomorphic {Euler-Poincar\'e-characteristic}.

\bigskip
\noindent
{\bf Part 1:} 
In the first part we determine the set of \emph{admissible numerical data}. This is the finite set 
of tuples of the form 
\begin{itemize}
\item 
$(n,T_1,T_2)$ in the index two case and 
\item
$(n,T_1)$ in the index three and index six case,
\end{itemize}
such that the combinatorial constraints form Proposition \ref{bounds1} and Remark \ref{roots}, 
the inequality from Proposition 
\ref{boundgroupord} and Hurwitz' formula are satisfied. 

\noindent
Note that the set of numerical data of threefolds isogenous to a product with 
$\chi(\mathcal O_X)=\chi$ is 
a subset of the set of admissible numerical data. 

\noindent
In our implementation, this computation is performed by the functions
{\tt AdNDindexTwo}, {\tt AdNDindexThree} and {\tt AdNDindexSix} in the respective cases. 
The functions just return the set of admissible numerical data such that the groups of order $n$ in the unmixed case, 
$n/2$ in the index two case and $n/3$ in the index three and index six case are contained in 
the Database of Small Groups. The exceptions are stored in the files 
{\tt ExcepIndexTwo$\chi$.txt},  {\tt ExcepIndexThree$\chi$.txt} and {\tt ExcepIndexSix$\chi$.txt}. 

\bigskip
\noindent
{\bf Part 2:} 
In the second part of the algorithm, we search for algebraic data.

\bigskip
\noindent
\underline{{\bf Index two case:}}

\bigskip
\noindent
{\bf Step 1:} 
Starting from the triples $(n,T_1,T_2)$ contained in the set {\tt AdNDindexTwo($\chi$)}, 
compute the set of $4$-tuples $(n,T_1,T_2,H)$, where $H$ is a group of order $n/2$ admitting at least one 
generating vector of type $T_2$. 

\noindent
In our implementation, this computation is performed by the function
{\tt NDHIndexTwo}. The set of $4$-tuples $(n,T_1,T_2,H)$ such that the groups of 
order $n$ are contained in the Database of Small Groups is returned.
The remaining tuples are stored in the file {\tt ExcepIndexTwo$\chi$.txt}. 

\noindent
{\bf Step 2:} 
For each integer $n$  belonging  to some $4$-tuple in the set {\tt NDHIndexTwo($\chi$)}
consider the groups of order $n$.
For each group $G$ of order $n$ construct  the list of subgroups of index two. 
For each $G^0$ in this list consider the $4$-tuples $(n,T_1,T_2,H)$ from Step 1 such that 
$H \simeq G^0$. For each of this $4$-tuples 
compute the set of generating vectors $V_1$ for $G$ of type $T_1$ and the set of generating vectors 
$V_2$ for $G^0$ of type $T_2$. Check the freeness conditions $i)$ and $ii)$ of Proposition \ref{freenesscond} $b)$.
 If they are fulfilled, then there exists a threefold $X$ isogenous to a product
with algebraic datum $(G,G^0,V_1,V_2)$ and $\chi(\mathcal O_X)=\chi$ (see Proposition \ref{converse}). 
Compute the Hodge diamond of $X$  
and save the occurrence 
\[
[G,T_1,T_2, h^{3,0}, h^{2,0}, h^{1,0}, h^{1,1}, h^{1,2}]
\]
in the file {\tt IndexTwo$\chi$.txt}. Step 2 is performed calling {\tt ClassifyIndexTwo($\chi$)}. 

\bigskip
\noindent
\underline{{\bf Index three case:}} 

\bigskip
\noindent
{\bf Step 1:} 
Starting from the pairs $(n,T_1)$ contained in the set {\tt AdNDindexThree($\chi$)}, compute the set of triples 
$(n,T_1,H)$, where $H$ is a group of order $n/3$ admitting three generating vectors $V_1$,$V_1'$ and $V_1''$ 
of type $T_1$ such that the associated stabilizer sets $\Sigma_1$, $\Sigma_1'$ and $\Sigma_1''$ fulfill the condition 
\[
\Sigma_1 ~ \cap ~ \Sigma_1' ~ \cap ~ \Sigma_1''  = \lbrace 1_{H} \rbrace.
\]
Here we use the fact that a threefold isogenous to a product of mixed type with numerical datum $(n,T_1)$ is 
covered by a threefold of unmixed type, where $|G^0|=n/3$. 

\noindent
In our implementation, this computation is performed by the function
{\tt NDHIndexThree}. The set of triples $(n,T_1,H)$ such that the groups of 
order $n$ are contained in the Database of Small Groups is returned.
The remaining triples are stored in the file {\tt ExcepIndexThree$\chi$.txt}. 

\bigskip
\noindent
{\bf Step 2:} 
For each integer $n$  belonging to a triple from Step 1 consider the groups of order $n$.
For each group $G$ of order $n$ construct the list of normal subgroups $G^0$ of index three 
such that the short exact sequence 
\[
1 \to G^0 \to G \to \mathfrak A_3 \to 1 
\]
does not split.
For each $G^0$ in this list consider the triples $(n,T_1,H)$ from Step 1 such that $H \simeq G^0$. 
For each of these $4$-tuples choose an element $\tau \in G \setminus G^0$ and 
compute all generating vectors 
$V_1$ for $G^0$ of type $T_1$. 
Check the freeness condition $i)$ of Proposition \ref{freenesscond} $c)$.
If it holds, then the second condition of the proposition is also fulfilled, since the sequence 
\[
1 \to G^0 \to G \to \mathfrak A_3 \to 1 
\]
is non-split which is an equivalent condition according to Proposition \ref{shortexactnonsplit}.
Therefore, there exists a threefold $X$ isogenous to a product with algebraic datum 
$(G,G^0,\tau,V_1)$ and $\chi(\mathcal O_X)=\chi$ (see Proposition \ref{converse}). 
Compute the Hodge diamond of $X$  
and save the occurrence 
\[
[G,T_1,h^{3,0}, h^{2,0}, h^{1,0}, h^{1,1}, h^{1,2}]
\]
in the file {\tt IndexThree$\chi$.txt}. Step 2 is performed calling {\tt ClassifyIndexThree($\chi$)}. 

\bigskip
\noindent
\underline{{\bf Index six case:}} 

\bigskip
\noindent
{\bf Step 1:} 
Starting from the pairs $(n,T_1)$ contained in the set {\tt AdNDindexSix($\chi$)}, compute the set of triples 
$(n,T_1,H)$, where $H$ is a group of order $n/3$ admitting a generating vector $V_1$ of type $T_1$. 

\noindent
In our implementation, this computation is performed by the function
{\tt NDHIndexSix}. The set of triples $(n,T_1,H)$ such that the groups of 
order $n$ are contained in the Database of Small Groups is returned.
The remaining triples are stored in the file {\tt ExcepIndexSix$\chi$.txt}.

\noindent
{\bf Step 2:} 
For each integer $n$  belonging to a triple from Step 1 consider the list of groups of order $n$.
For each group $G$ of order $n$, consider the triples of the form 
$(n,T_1,H)$ such that $G$ admits a subgroup of index three isomorphic to $H$. 
Compute the set of normal subgroups $G^0$ of $G$ of index six such that the short exact sequence 
\[
1 \to G^0 \to G \to \mathfrak S_3 \to 1 
\]
does not split. 
Choose elements $\tau,h \in G \setminus G^0$ such that $\tau^2 \notin G^0$ and $h^2 \in G^0$. 
If the group $G_1:=G^0 \cup h \cdot G^0$ is isomorphic to $H$, then compute all generating vectors $V_1$ of type 
$T_1$ for this group. For each of these vectors compute the associated stabilizer set $\Sigma_1$ and 
check the freeness conditions $i)$, $ii)$ and $iii)$ of Proposition \ref{freenesscond} $d)$.
If they are fulfilled, then there exists a threefold $X$ isogenous to a product 
with algebraic datum $(G,G^0,\tau,h,V_1)$ and $\chi(\mathcal O_X)=\chi$ (see Proposition \ref{converse}). 
Compute the Hodge diamond of $X$  
and save the occurrence 
\[
[G,T_1,h^{3,0}, h^{2,0}, h^{1,0}, h^{1,1}, h^{1,2}]
\]
in the file {\tt IndexSix$\chi$.txt}. Step 2 is performed calling {\tt ClassifyIndexSix($\chi$)}.

\bigskip

\begin{comprem}\label{except}
\makebox{$$}

\begin{itemize}
\item
In Part 2 of the algorithm we search for generating vectors. We point out that different generating vectors may 
determine threefolds with the same invariants. For example,
this happens if (but not only if) they differ by some \textit{Hurwitz moves}.
These moves are described in \cite{CLP15}, \cite{Zimmermann} and \cite{pen10} and we refer to these sources for
further details.
\item
We point out that for a generating vector of type $[g';-]$ the stabilizer set 
is trivial and the corresponding character $\chi_{\varphi}$ is the sum of the trivial character 
and $(g'-1)$ copies of the regular character according to the 
formula of Chevalley-Weil  (see \cite{ChevWeil}). 
Consequently, in this case it is sufficient for us to know the existence of a generating vector, but
there is no need to compute all of them. 
\end{itemize}
\end{comprem}

\bigskip

\noindent
{\bf Main Computation.}

\bigskip
\noindent
We execute the implementation for the input value $\chi=-1$. 
Note that the 
combinatorial constraints in Part 1 of the program are very strong, so
relatively few admissible numerical data are returned.
The total number of admissible group orders turns out to be relatively small and the 
maximum possible group order drops significantly compared to the theoretical bound from 
Proposition \ref{boundgroupord}. 
The table below summarizes the occurrences: 

\bigskip

{\small
\begin{center}
\setlength{\tabcolsep}{1pt}
\begin{tabular}{|c |c | c | c | c|}
\hline
$$ &  ~ index two ~ & ~ index three ~ & ~ index six ~  \\
\hline 
No. AdNumData &  $253$ & $8$ & $5$  \\
\hline 
No. G-Orders & $39$ & $2$ & $1$  \\
\hline 
$n_{max}$ &  $576$ & $216$ & $216$  \\
\hline
$n_{theo}$ &  $1539$ & $4000$ & $4000$  \\
\hline

\end{tabular}
\end{center}
}

\bigskip
\noindent 
In the first row we report the total number of admissible numerical data,
in the second row the total number of group orders,
in the third row the maximum possible group order after performing Part 1 of the algorithm and in the last row 
the theoretical bound for the group order according to Proposition 
\ref{boundgroupord}. There are no exceptional numerical data to be considered, i.e. the files 
{\tt ExcepIndexTwo$\chi$.txt},  {\tt ExcepIndexThree$\chi$.txt} and {\tt ExcepIndexSix$\chi$.txt} remain empty. 
The table below reports the computation time to run the complete program (Part 1 and Part 2) on a 
$8 \times 2.5$GHz Intel Xenon L5420 workstation with 16GB RAM in the respective cases:

\bigskip

{\small
\begin{center}\label{cputime} 
\setlength{\tabcolsep}{1pt}
\begin{tabular}{|c |c | c | c | c|}
\hline
$$ & ~ index two ~ & ~ index three ~ & ~ index six ~  \\
\hline
time &  10h  28min &  24 sec & 30 sec \\
\hline
\end{tabular}
\end{center}
}

\bigskip
\noindent 
This computation yields our main result:
the classification of 
threefolds isogenous to a product of mixed type 
with $\chi(\mathcal O_X)=-1$ and absolutely faithful $G$-action 
(see Theorem \ref{indexzwei} and Theorem \ref{indexdreiundsechs}).

\begin{comprem}
Running Part 1 of the program for values of $\chi$ different from $-1$  
exceptional numerical data might occur. 
We tried and executed Part 1 in the index two, index three and index six case for all values of $\chi$ in the range 
\[
-40 \leq \chi \leq -1
\]
and found no exceptional numerical data.
Albeit it is not of great importance in our context, we shall mention that there are methods to deal with the exceptional numerical data, if they should occur for $\chi \leq -41$. 
We refer the reader to the paper \cite{bauer}, where the authors classify surfaces isogenous to a product 
with $p_g=q=0$ and the analogous problem appears. Their strategy can be easily adapted to the threefold case. 
Nevertheless, running Part $2$ of the program for $\chi$ different from $-1$ is very time and memory consuming, in particular in the index two case: when we decrease $\chi$, then the maximal possible value for $g_i'$ 
increases, according to Proposition \ref{bounds1} vi). Similarly, the maximal length $r_i$ of the types 
\[
T_i=[g_i';m_{i,1}, \ldots ,m_{i,r_i}]
\]
that we obtain increases as well. This leads to a large number of generating vectors to be analysed, 
which slows down the computation and requires a lot of memory. 
\end{comprem}

\noindent
To conclude this chapter we give two further examples of threefolds $X$ isogenous to a product
with $\chi(\mathcal O_X)=-1$. The first one is of mixed type, obtained by an index six action,
the second one is of unmixed type without parameters, i.e. a rigid example. 
Note that we have no index six example and no rigid examples in the absolutely faithful case with 
$\chi(\mathcal O_X)=-1$ neither of mixed nor unmixed type (see \cite[Theorem 0.1]{FG16}). 
Therefore, to produce these examples, we have to allow non-trivial kernels. 

\bigskip

\begin{example}\label{Index6Examples}
\makebox{$$}
\end{example}
\noindent
We begin with the index six example. Consider the group
$G:=SmallGroup(216,90)$, it admits a unique normal subgroup $G^0$ such that $G/G^0 \simeq \mathfrak S_3$.
Moreover, the extension 
\[
1 \to G^0 \to G \to \mathfrak S_3 \to 1 
\]
is non-split. 
For the elements $h:=G.1 * G.2 * G.4^2$ and $\tau:=G.3 * G.4^2$ in $G\setminus G^0$ it holds 
\[
\tau^2 \notin G^0 \quad \quad \makebox{and} \quad \quad h^2 \in G^0,
\]
i.e. $h$ and $\tau$ define an isomorphism $G/G^0 \to \mathfrak S_3$. 
The cyclic group $K_1$ generated by $G1.3 * G1.4$ is the unique normal subgroup  
 in $G_1:=\langle h,G^0 \rangle$ 
of order six such that 
\[
K_1 \cap \tau K_1 \tau^{-1} = \lbrace 1_G \rbrace.
\] 
The quotient 
$G_1/K_1$ is isomorphic to the dihedral group
$\mathcal D_6$ via the map $G_1/K_1 \to  \mathcal D_6$ defined by 
\[
\overline{G1.1} \mapsto s \quad \quad \makebox{and} \quad \quad \overline{G1.2 * G1.5} \mapsto t.
\]
There is a faithful group action $\mathcal D_6 \to \Aut(C)$, where $C$ is a compact Riemann 
surface of genus $g(C)=7$. A corresponding generating vector is given by $V_1:=(st,st,t^5,t^5)$. 
The stabilizer set $\Sigma_1$ of the action 
\[
\psi_1 \colon G_1 \to G_1/K_1 \simeq \mathcal D_6 \to \Aut(C)
\]
fulfills the freeness conditions:

\begin{itemize}
\item[i)]
$\Sigma_1 ~ \cap ~ \tau \Sigma_1 \tau^{-1} ~ \cap ~ \tau^2 \Sigma_1 \tau^{-2} = \lbrace 1_G \rbrace$.
\item[ii)]
$(\tau g)^3 \notin \Sigma_1$ for all  $g \in G^0$  and  
\item[iii)]
$\tau f^2 \tau^{-1} \notin \Sigma_1$ for all $f \in G_1 \setminus G^0 \cap \Sigma_1$.
\end{itemize}
\noindent
According to Proposition \ref{converse} c), the tuple $(G,G^0,K_1,\tau,h,V_1)$ is an algebraic datum of a threefold 
$X=C^3/G$ isogenous to a product. Since $g(C)=7$, it holds 
\[
\chi(\mathcal O_X)=-\frac{(g(C)-1)^3}{216}=-1.
\]
For completeness, we also determine the Hodge numbers (cf. Section \ref{hodgenumbers}): 
\[
h^{3,0}(X)= 2, \quad h^{2,0}(X)= 1, \quad  h^{1,0}(X)= 1, \quad h^{1,1}(X)= 5 
\quad \makebox{and} \quad  h^{1,2}(X)= 8. 
\]

\begin{example}\label{RigidExample}
\makebox{$$}
\end{example}
\noindent
Let $S=\big(C_1 \times C_2\big)/G$ be a rigid surface isogenous to a product of unmixed type, 
then $C_i/G \simeq \mathbb P^1$ and the $G$-covers $C_i \to \mathbb P^1$ are branched over 
$0,1$ and $\infty$. These surfaces are called Beauville surfaces, since 
Beauville provided the first example of such a surface (cf. \cite{Be83}). 
In his example $G=\mathbb Z_5^2$ and $g(C_i)=6$ yielding $\chi(\mathcal O_S)=1$.
Appropriate generating vectors $V_i$ for $\mathbb Z_5^2$ are given by 
\[
V_1=\big[ (0,3), (3,3), (2,4) \big]   \quad \quad \makebox{and} \quad \quad V_2=\big[ (2,0), (2,1), (1,4) \big].
\]
We can easily modify this example to obtain a rigid threefold isogenous to a product with $\chi(\mathcal O_X)=-1$.
Consider the generating vector $V_3=(1,1,3)$ of $\mathbb Z_5$. It corresponds to an action 
\[
\psi_3 \colon \mathbb Z_5 \to \Aut(C_3),
\]
where $C_3$ is a curve of genus two, $C_3/\mathbb Z_5 \simeq \mathbb P^1$ and the $\mathbb Z_5$-cover
$C_3 \to \mathbb P^1$ is branched over $0,1$ and $\infty$. 
We obtain a diagonal, free action of $\mathbb Z_5^2$ on $C_1 \times C_2 \times C_3$, where 
$\mathbb Z_5^2$ acts on $C_3$ via $\psi_3$ composed with the projection to the first factor.  
The quotient 
\[
X=\big(C_1 \times C_2 \times C_3 \big)/\mathbb Z_5^2
\]
is a rigid threefold isogenous to a product with $\chi(\mathcal O_X)=-1$ and the Hodge numbers of $X$ are the following: 
\[
h^{3,0}(X)= 3, \quad h^{2,0}(X)= 1, \quad  h^{1,0}(X)= 0, \quad h^{1,1}(X)= 5 \quad \makebox{and} 
\quad  h^{1,2}(X)= 9. 
\]

\bigskip
\bigskip

\noindent {\bf Author's Address.} \\
\emph{Christian Glei\ss ner:} Dipartimento di Matematica, 
Universit\`{a} degli Studi di Trento; \\
Via Sommarive 14; I-38123 Povo (Trento), Italy

\end{document}